\documentclass[12pt]{amsart}
\usepackage[utf8]{inputenc}
\usepackage[colorlinks=true]{hyperref}
\usepackage{amsmath,amscd,amsbsy,amssymb,amsthm,amsfonts,array}
\usepackage{pgfplots}
\usepackage{latexsym,url,bm, verbatim}
\usepackage{cite,enumitem}
\usepackage{ulem,xcolor}
\usepackage{tikz}

\pgfplotsset{compat=1.18}

\newcommand{\nP}{{\mathbf P}}

\usepackage{fullpage}

\newtheorem{theorem}{Theorem}[section]
\newtheorem{lemma}[theorem]{Lemma}
\newtheorem{proposition}[theorem]{Proposition}
\newtheorem{corollary}[theorem]{Corollary}
\newtheorem{definition}[theorem]{Definition}

\theoremstyle{remark}
\newtheorem{remark}[theorem]{Remark}
\numberwithin{equation}{section}
\numberwithin{figure}{section}
\allowdisplaybreaks

\title{Existence of two-dimensional deep capillary solitary water waves with constant vorticity}

\author{James Rowan}
\address{Department of Mathematics, University of Texas at Austin}
\curraddr{}
\email{james.rowan@austin.utexas.edu}

\author{Lizhe Wan}
\address{Beijing International Center for Mathematical Research, Peking University}
\curraddr{}
\email{wanlizhe@pku.edu.cn}

\begin{document}

\begin{abstract}
 We consider two-dimensional capillary water waves with nonzero constant vorticity in infinite depth.
  We first derive the Babenko equation that describes the profile of the solitary wave.
  When the velocity $c$ is close to a critical velocity, the Babenko equation can be reduced to the stationary focusing cubic nonlinear Schr\"odinger equation plus perturbative error.
  This gives the existence of small-amplitude solitary wave solutions for the water wave system with constant vorticity whose profiles are close to solutions of the stationary focusing cubic nonlinear Schr\"odinger equation.
\end{abstract}

\maketitle

\section{Introduction} \label{s:Intro}
We consider the two-dimensional  water waves with constant vorticity on deep water without viscosity.
The fluid occupies a time-dependent domain $\Omega_t \subset \mathbb{R}^2$ which has infinite depth and a free upper boundary $\Gamma_t$ which is asymptotically flat at $y \approx 0$.
Denote the fluid velocity by $\mathbf{u}(t,x,y) = (u(t,x,y), v(t,x,y))$, the pressure by $p(t,x,y)$, and the constant vorticity by $\gamma$. 
The equations inside $\Omega_t$ are
\begin{equation*}
\left\{
             \begin{array}{lr}
            u_t +uu_x +vu_y = -p_x &  \\
            v_t + uv_x +vv_y = -p_y -g& \\
            u_x +v_y =0 & \\
            \omega := u_y -v_x = -\gamma.
             \end{array}
\right.
\end{equation*}
On the boundary $\Gamma_t$ we have the dynamic boundary condition
\begin{equation*}
    p = -\sigma\mathbb{H},
\end{equation*}
and the kinematic boundary condition
\begin{equation*}
    \partial_t +\mathbf{u}\cdot \nabla \text{ is tangent to }\Gamma_t.
\end{equation*}
Here $g$ is the gravitational constant and $\sigma$ represents the strength of the surface tension.
They are nonnegative parameters, and at least one of them is non-zero. 
Let $\eta(x)$ be the graph of the free upper boundary $\Gamma_t$.
$\mathbb{H}$ is the mean curvature of the free boundary and can then be expressed by 
\begin{equation*}
   \mathbb{H}(\eta) = \left(\frac{\eta_x}{\sqrt{1+\eta_x^2}}\right)_x.
\end{equation*}

 \begin{figure}
   \centering

   \begin{tikzpicture}
     \begin{axis}[ xmin=-12, xmax=12.5, ymin=-2.5, ymax=1.5, axis x
       line = none, axis y line = none, samples=100 ]

       \addplot+[mark=none,domain=-10:10,stack plots=y]
       {0.2*sin(deg(2*x))*exp(-x^2/20)};
       \addplot+[mark=none,fill=gray!20!white,draw=gray!30,thick,domain=-10:10,stack
       plots=y] {-1.5-0.2*sin(deg(2*x))*exp(-x^2/20)} \closedcycle;
      \addplot+[black, thick,mark=none,domain=-10:10,stack plots=y]
       {1.5+0.2*sin(deg(2*x))*exp(-x^2/20)};

       \draw[->] (axis cs:-3,-2) -- (axis cs:-3,1) node[left] {\(y\)};
       \draw[->] (axis cs:-11,0) -- (axis cs:12,0) node[below] {\(x\)};
       \filldraw (axis cs:-3,-1.5)  node[above left]
       {\(\)}; \node at (axis cs:0,-0.75) {\(\Omega(t)\)}; \node at
       (axis cs:6,0.3) {\(\Gamma(t)\)};

    \end{axis}
  \end{tikzpicture}
   \caption{The fluid domain.}
 \end{figure}

The water wave system is often studied in the celebrated Zakharov-Craig-Sulem formulation; see, for example, Wahl\'en \cite{MR2309783}.
In this paper, we reexpress the water wave system using a different framework.
In the zero surface tension case, the above system  was studied using holomorphic position/velocity potential variables $(W(t,\alpha),Q(t,\alpha))$ by  Ifrim-Tataru \cite{MR3869381} and  Ifrim-Rowan-Tataru-Wan \cite{MR4462478}.
In the presence of nonzero surface tension, the water wave system with constant vorticity has the following form:
\begin{equation}
\left\{
             \begin{array}{lr}
             W_t + (W_\alpha +1)\underline{F} +i\dfrac{\gamma}{2}W = 0 &  \\
             Q_t - igW +\underline{F}Q_\alpha +i\gamma Q +\mathbf{P}\left[\dfrac{|Q_\alpha|^2}{J}\right]- i\dfrac{\gamma}{2}T_1-2\sigma \mathbf{P}\Im\left[ \dfrac{W_{\alpha \alpha}}{J^{\frac 1 2}(1+W_{\alpha})}\right]  =0,&  
             \end{array}
\right.\label{e:CVWW}
\end{equation}
where $J := |1+ W_\alpha|^2$, $\mathbf{P}$ is the projection onto negative frequencies, namely
\begin{equation*}
    \mathbf{P} = \frac{1}{2}(\mathbf{I} - iH),
\end{equation*}
with $H$ denoting the Hilbert transform, and
\begin{equation*}
\begin{aligned}
&F: = \mathbf{P}\left[\frac{Q_\alpha - \Bar{Q}_\alpha}{J}\right], \quad &F_1 = \mathbf{P}\left[\frac{W}{1+\Bar{W}_\alpha}+\frac{\Bar{W}}{1+W_\alpha}\right],\\
&\underline{F}: =F- i \frac{\gamma}{2}F_1,  \quad &T_1: = \mathbf{P}\left[\frac{W\Bar{Q}_\alpha}{1+\Bar{W}_\alpha}-\frac{\Bar{W}Q_\alpha}{1+W_\alpha}\right].
\end{aligned}
\end{equation*}
We refer interested readers to Section $1$ of  Ifrim-Tataru \cite{MR3667289} and  Appendix $B$ of Ifrim-Tataru \cite{MR3869381} for the derivation of the water wave system \eqref{e:CVWW} in the capillary and the constant vorticity setting, respectively; this equation is a combination of the terms from those two effects.
We also remark that the relationship between the holomorphic coordinates and  Zakharov-Craig-Sulem formulation is discussed in Section $3.6$ of Alazard-Ifrim-Tataru \cite{MR4436142}.

A well-posedness result in the periodic setting for the constant vorticity water wave system with both gravity and surface tension is due to Berti-Maspero-Murgante \cite{MR4658635}.
They showed the almost global-in-time existence of small amplitude solutions for \eqref{e:CVWW} for almost all choices of positive $\sigma$.
For the well-posedness result on the real line, Wan \cite{MR4891579} showed that \eqref{e:CVWW} is locally well-posed if the initial data satisfy the condition $\left(\partial_\alpha W_{0}, \frac{\partial_\alpha Q_{0}}{1+\partial_\alpha W_{0}}\right) \in H^{s+\frac{1}{2}}(\mathbb{R})\times H^s(\mathbb{R})$ for $s>\frac{5}{4}$.

The system \eqref{e:CVWW} does not have a complete spacetime scaling symmetry.
As a consequence, it is impossible to rescale $g, \gamma$, and $\sigma$ to $1$ at the same time.
We thus keep them in this article as fixed parameters  to help organize terms of various scalings, except in Section \ref{s:parameterRegimes} where we nondimensionalize the equations using the dimensionless parameter $V=\frac{\sigma\gamma^4}{g^3}$ that measures the relative strength of the vorticity.

The zero solution is a trivial solution of \eqref{e:CVWW}, but we are interested in the existence and other properties of nontrivial solutions.
A natural question in any water wave model is the existence and properties of solitary wave solutions.
In other words, are there any nontrivial solutions of \eqref{e:CVWW} having the form
\begin{equation}
    \left(W(t,\alpha), Q(t,\alpha)\right) = \left(W(\alpha + ct), Q(\alpha + ct)\right), \qquad \lim_{\alpha \rightarrow \infty} (W(\alpha), Q(\alpha))= (0,0) \label{Form}
\end{equation}
for $t\geq 0$? 

We remark that the requirement that the solution vanishes at infinity excludes the possibility of periodic or quasi-periodic traveling wave solutions.
These types of special solutions of \eqref{e:CVWW} were studied by Wahl\'en \cite{MR2262949}, Martin \cite{MR2969824} and Berti-Franzoi-Maspero  \cite{MR4228858}.

To our knowledge, the known historical results for the existence/nonexistence of deep solitary waves with constant vorticity in two space dimensions are summarized in Table \ref{table:1} below.
\begin{table}[htb]   
\begin{center}   
\caption{Existence/Nonexistence of 2D deep solitary waves with constant vorticity}  
\label{table:1} 
\begin{tabular}{|m{2.3cm}<{\centering}|m{3.6cm}<{\centering}|m{3.8cm}<{\centering}|m{3.8cm}<{\centering}|}   
\hline   \textbf{Gravity} & \textbf{Surface tension} & \textbf{Constant vorticity} & \textbf{Existence} \\   
\hline   Yes & No & Zero  & No \cite{MR2993054, MR4151415} \\ 
\hline  No & Yes & Zero  & No \cite{MR4151415} \\  
\hline   Yes & Yes & Zero  & Yes  \cite{MR2069635, MR2847283, MR1423002}\\ 
\hline   Yes & No & Nonzero  & Yes \cite{rowan2023dimensional} \\
\hline   No & Yes & Nonzero  & Subject of this work\\
\hline   Yes & Yes & Nonzero  & Subject of this work \\
\hline   
\end{tabular}   
\end{center}   
\end{table}

In the case of zero vorticity $\gamma = 0$, it was shown by Hur \cite{MR2993054} and Ifrim-Tataru  \cite{MR4151415} that no solitary waves exist for  pure gravity cases. 
Ifrim-Tataru further showed that no pure capillary solitary wave exists in \cite{MR4151415}.
When both gravity and surface tension are present, solitary waves are proved to exist by Iooss-Kirrmann \cite{MR1423002}, Buffoni \cite{MR2069635}, and Groves-Wahl\'{e}n \cite{MR2847283}.

Recently, the authors constructed  deep pure gravity solitary waves in the presence of nonzero constant  vorticity~\cite{rowan2023dimensional}.
In addition to the existence result, an asymptotic description of the solitary wave profile was given, leveraging the simpler algebraic structure of working in holomorphic coordinates.
This article studies the capillary case.

We note the series of works~\cite{MR3849609,MR3961580,MR4917706} which perform multi-scale perturbative analysis of the water wave problem with constant vorticity and analyze NLS-like solitary waves in various regimes, including their stability in the three-dimensional setting under small transverse perturbations.
In contrast to the construction in~\cite{MR3961580}, we do not treat the vorticity as a parameter in a perturbative expansion, instead handling all nonzero constant vorticities equally.
Our construction relies on the theory of quasilinear elliptic equations and gives quantitative error estimates in Banach spaces.

To give an  overview of the broader history of 2D solitary water waves, we also recall some historical results for the finite depth case here.
In the case of zero vorticity $\gamma = 0$,  the existence of pure gravity solitary  waves in finite depth dates back to work of Lavrentyev~\cite{MR61952} and Friedrichs-Hyers \cite{MR65317}, followed by results of Beale \cite{MR445136}, Amick-Toland \cite{MR629699} and Plotnikov \cite{MR1133302}.
As for the non-gravity case,
Ifrim-Pineau-Tataru-Taylor showed in \cite{MR4455193} that there is no solitary wave for the pure capillary finite depth case.  
When both gravity and surface tension are present, solitary waves are proven to exist; see Amick-Kirchg\"{a}ssner \cite{MR963906}, Iooss-Kirchg\"{a}ssner \cite{MR1071623}, Buffoni \cite{MR2073504, MR2133392}, Buffoni-Groves \cite{MR1378603} and Groves \cite{MR4246394}.
In the case of nonzero constant vorticity, Kozlov, Kuznetsov, and Lokharu proved the existence of solitary waves in the pure-gravity case in \cite{MR4164801}.
Groves and Wahl\'{e}n showed in \cite{MR3415532} the existence of solitary waves in the gravity-capillary setting for small momentum.

In the case of general vorticity in the presence of gravity and finite depth, solitary waves have long been known to exist~\cite{terKrikorov61}, and recent work has shown the existence of small solitary waves~\cite{grovesWahlen08general} using a center manifold reduction, explicit formulas for the waves~\cite{hur08exact}, and the existence of large solitary waves~\cite{wheeler13large}.
As for solitary waves for water waves with constant vorticity, there are a number of results on existence and stability in the gravity-capillary case, especially in finite depth, see for instance \cite{MR3415532, MR1364730, MR1297863}; solitary waves have been constructed in the pure gravity finite depth case as well~\cite{haziotWheeler23}.

In this article, we prove the existence of solitary waves for the remaining 2D cases, the deep gravity-capillary case and the deep pure capillary case with nonzero constant vorticity.
Before stating the exact existence result, we give some heuristics by looking at the linearization of water waves \eqref{e:CVWW} around the zero solution.
The linearized system is given by 
\begin{equation*}
\left\{
             \begin{array}{lr}
             w_t + q_\alpha = 0 &  \\
             q_t +i\gamma q - igw +i\sigma w_{\alpha \alpha} =0,&  
             \end{array}
\right.
\end{equation*}
restricted to holomorphic functions, i.e. complex-valued functions whose spatial frequencies are non-positive following the notation in \cite{MR3667289, MR3869381}.
By eliminating the $q$ variable, the linearized system can be written as a second order linear dispersive equation
\begin{equation} \label{e:ZeroLinear}
    w_{tt}+ i \gamma w_t + igw_\alpha - i\sigma w_{\alpha \alpha \alpha} =0.
\end{equation}
Its dispersion relation is given by
\begin{equation*}
    \tau^2 + \gamma\tau +g\xi +\sigma \xi^3 =0, \quad \xi\leq 0,
\end{equation*}
where $\tau$ and $\xi$ represent the temporal and spatial Fourier variables.
Suppose the equation \eqref{e:ZeroLinear} has a linear wave solution of frequency $k< 0$ and velocity $c$, i.e., $e^{ik(\alpha+ ct)}$; then the frequency $k$ must satisfy the relation
\begin{equation*}
    \sigma k^2 +c^2 k + (g+c\gamma) =0, \quad k< 0.
\end{equation*}
Viewing the velocity $c$ as a function of frequency $k$, we get
\begin{equation}
    c^\pm(k) = \dfrac{-\gamma \pm \sqrt{\gamma^2 -4k(g+\sigma k^2)}}{2k}, \quad k< 0. \label{Defcpm}
\end{equation}
A rough graph of $c^{\pm}(k)$ is given in Figure \ref{f:Cplusminus}.

Direct computation gives that $c^{+}(k)$ is negative, with a unique global maximum  $c^{+}(\omega_1) = c_1$ and $c^{-}(k)$ is positive, with a unique global minimum  $c^{-}(\omega_2) = c_2$, where $c_1 < 0< c_2$ are the only two real roots of the equation of the discriminant
\begin{equation}
c^4 = 4 \sigma (g+c\gamma), \quad g+c\gamma>0. \label{RangeOfC}
\end{equation}
Note that at these two critical velocities,
\begin{equation*}
 \sigma k^2 +c_1^2 k + (g+c_1\gamma) \geq 0, \quad  \sigma k^2 +c_2^2 k + (g+c_2 \gamma) \geq 0.   
\end{equation*}
These two inequalities become equalities only at $\omega_1$ and $\omega_2$ respectively.
Hence, we have $c_j^2=2\sigma|\omega_j|$.
 \begin{figure} 
  \centering

  \begin{tikzpicture}
     \begin{axis}[ xmin=-6.5, xmax=1.5, ymin=-6.5, ymax=6.5, axis x
       line = none, axis y line = none, samples=200 ]
       \addplot+[mark=none,draw=black,domain=-5:-2.5]
       {(-1+(1-4*1*x*(1+2*x^2))^(1/2))/(2*1*x)} node[below,text=black] {\(c^+(\omega)\)};
       \addplot+[mark=none,draw=black,domain=-2.5:-0.01]
       {(-1+(1-4*1*x*(1+2*x^2))^(1/2))/(2*1*x)};
        \addplot+[mark=none,domain=-5:-2.5]
       {(-1-(1-4*1*x*(1+2*x^2))^(1/2))/(2*1*x)} node[above] {\(c^-(\omega)\)};
       \addplot+[mark=none,domain=-2.5:-0.01]
       {(-1-(1-4*1*x*(1+2*x^2))^(1/2))/(2*1*x)};
       \draw [dashed] (axis cs:-0.20844,-0.913) -- (axis cs:-0.20844,0) node[above] {\(\omega_1\)};
       \draw [dashed] (axis cs:-1.21,2.25) -- (axis cs:-1.21,0) node[below] {\(\omega_2\)};
       \draw[->] (axis cs:0,-5.5) -- (axis cs:0,5.5) node[right] {\(c\)};
       \draw[->] (axis cs:-5.5,0) -- (axis cs:0.5,0) node[below] {\(k\)};
     \end{axis}
   \end{tikzpicture}
   \caption{The graphs of $c^{\pm}(k)$ in the case $g>0$, with critical points at frequencies $\omega_1$ and $\omega_2$.
   The parameters $\gamma$ and $g$ are set equal to $1$, and $\sigma$ is set equal to $2$ in this figure.} \label{f:Cplusminus}
 \end{figure} 

 \begin{figure} 
  \centering

  \begin{tikzpicture}
     \begin{axis}[ xmin=-6.5, xmax=1.6, ymin=-6.5, ymax=6.5, axis x
       line = none, axis y line = none, samples=200 ]
       \addplot+[mark=none,draw=black,domain=-5:-2.5]
       {(-1+(1-4*1*x*(2*x^2))^(1/2))/(2*1*x)} node[below,text=black] {\(c^+(\omega)\)};
       \addplot+[mark=none,draw=black,domain=-2.5:-0.01]
       {(-1+(1-4*1*x*(2*x^2))^(1/2))/(2*1*x)};
        \addplot+[mark=none,domain=-5:-2.5]
       {(-1-(1-4*1*x*(2*x^2))^(1/2))/(2*1*x)} node[above] {\(c^-(\omega)\)};
       \addplot+[mark=none,domain=-2.5:-0.01]
       {(-1-(1-4*1*x*(2*x^2))^(1/2))/(2*1*x)};
       \draw [dashed] (axis cs:-0,0) -- (axis cs:0,0) node[anchor=south west] {\(\omega_1=0\)};
       \draw [dashed] (axis cs:-0.95,2.02) -- (axis cs:-0.95,0) node[below] {\(\omega_2\)};
       \draw[->] (axis cs:0,-5.5) -- (axis cs:0,5.5) node[right] {\(c\)};
       \draw[->] (axis cs:-5.5,0) -- (axis cs:0.5,0) node[below] {\(k\)};
     \end{axis}
   \end{tikzpicture}
   \caption{The graphs of $c^{\pm}(k)$ in the case $g=0$, with critical points at frequencies $\omega_1=0$ and $\omega_2$ in the zero-gravity case.
   The parameter $\gamma$ is set to $1$ and $\sigma$ is set to $2$ in this figure.} \label{f:Cplusminus0g}
 \end{figure} 

By Section $3$ of Dias and Kharif \cite{MR1670945}, when the linear group and phase velocities are equal so that $c'(k) = 0$, one can expect bifurcations of nonlinear solitary waves.
Therefore, we may expect the existence of small-amplitude solitary waves whose velocity is either $c_1+\epsilon^2$ or $c_2-\epsilon^2$ and bifurcate from a linear periodic wave  with frequency either $\omega_1$ or $\omega_2$.
Indeed, our main theorem below states that under such conditions, the solitary waves of \eqref{e:CVWW} exist, and gives a quantitative NLS approximation result.

\begin{theorem} \label{t:TheoremOne} 
Given any nonzero constant vorticity $\gamma$, positive capillary constant $\sigma$, and positive gravity constant $g$, let $c_1< 0<c_2$ be two roots of the equation \eqref{RangeOfC}.
If the velocity $c = c_1+\epsilon^2$ or $c = c_2 -\epsilon^2$ for some small constant $\epsilon$ and the frequency $\omega=\omega_1$ or $\omega=\omega_2$, respectively, then the system \eqref{e:CVWW} has two small  $H^\infty$ solitary wave solutions with velocity $c$.

Moreover, the holomorphic position variable $W$ is of the following type:
\begin{equation*}
     W(\alpha)  = \pm 2\epsilon i \rho_* (\epsilon \alpha) e^{i\omega \alpha} + o_{H^1}(\epsilon).
\end{equation*}
Here, $\pm \rho_*(\beta)$ are nontrivial real-valued even solutions of the stationary focusing cubic nonlinear Schr\"odinger (NLS) equation
\begin{equation}
    \left(a_1-\sigma \partial_{\beta}^2\right)\rho - a_2|\rho|^2 \rho =0, \label{CubicShrodingerEqn}
\end{equation}
where positive constants $a_1, a_2$ are defined in \eqref{AOneTwoDef}.
The explicit formula of $\rho_*$ is given in \eqref{FormulaRho}.
In addition, for the holomorphic velocity potential, $\Im Q$ is determined by \eqref{e:BabenkoQ}, and $Q = 2i\nP \Im Q$.
\end{theorem}

As a special case where the fluid is located in a place where the gravity is negligible or at a scale where the capillary effect dominates, we have the following result for pure capillary solitary waves with nonzero constant vorticity:
\begin{theorem} 
Let the constant vorticity $\gamma\neq 0$,   gravity constant $g=0$, and let $c_* = \sqrt[3]{4\sigma\gamma}$.
The system \eqref{e:CVWW} has two small $H^\infty$ solitary wave solutions of velocity  $c = c_*+\epsilon^2$ if $\gamma<0$ or $c = c_* -\epsilon^2$ if $\gamma>0$ for some small constant $\epsilon$.  
Moreover, $W$ is of the following type:
\begin{equation*}
     W(\alpha)  = \pm 2\epsilon i \rho_* (\epsilon \alpha)e^{i\omega \alpha} + o_{H^1}(\epsilon).
\end{equation*}
Here $\omega=-\left(2\gamma^2\sigma^{-1}\right)^{\frac 1 3}$ and the functions $\pm \rho_*(\beta)$ are  nontrivial real-valued even solutions of the stationary cubic focusing  NLS equation
\begin{equation*}
   \left (a_1-\sigma \partial_{\beta}^2\right)\rho - a_2|\rho|^2 \rho =0,
\end{equation*}
where positive constants $a_1, a_2$ are defined in \eqref{AOneTwoDef}.
The explicit formula of $\rho_*$ is given in \eqref{FormulaRho}.
In addition, $\Im Q$ is determined by \eqref{e:BabenkoQ}, and $Q = 2i\nP \Im Q$.
\end{theorem}

\begin{remark}
We make the following remarks about the main results.
\begin{enumerate}
\item To see why the equation \eqref{RangeOfC} has exactly two roots, we consider the function 
\begin{equation*}
   f(x) = x^4 -4\sigma \gamma x -4\sigma g. 
\end{equation*}
Then $f^{'}(x) = 4x^3 - 4\sigma\gamma$, so that $f(x)$ is decreasing on $(-\infty, \sqrt[3]{\sigma\gamma})$, and $f(x)$ is increasing on $(\sqrt[3]{\sigma\gamma}, \infty)$.
Also $f(0) = -4\sigma g <0$, and $f(x)$ is positive when $|x|$ is large.
  The range of velocity $c$ satisfies the sign condition
\begin{equation}
    g+c\gamma >0, \label{SignCon1}
\end{equation}
as well as the coercivity condition
\begin{equation}
   c^4<  4\sigma (g+c\gamma). \label{CoerciveCon1}
\end{equation}
These two conditions ensure that the Babenko equation \eqref{e:BabenkoEqnLR} that describes the profile of solitary waves is a second-order quasilinear elliptic equation.
\item In the zero surface tension case $\sigma =0$, the Babenko equation \eqref{e:BabenkoEqnLR} degenerates to a first-order elliptic equation.
The sign conditions for capillary and pure gravity  water waves with constant vorticity are different.
In the pure gravity case, we assume $g+c\gamma<0$ as in \cite{rowan2023dimensional}.
This is to be expected, due to the different convexities of the dispersion relations for the equations linearized around the zero solution.
In both cases, the sign condition ensures that the solitary wave velocity is not an allowed group velocity for linear dispersive waves.
\item In the zero gravity case $g=0$, the two real solutions of the equation $c^4=4\sigma c\gamma$ are $0$ and $\sqrt[3]{4\sigma\gamma}$.
The case $c=0$ does not correspond to a cubic NLS regime and thus we do not consider it in our analysis.
\item Our result also holds in the zero vorticity case $\gamma = 0$, where the profile of the solitary waves can be approximated by a  stationary focusing NLS equation, which is consistent with the result in Groves and Wahl\'{e}n \cite{MR2847283}.
\end{enumerate}
\end{remark}

Our main result on the existence of  solitary wave solutions of the two-dimensional gravity-capillary waves is based on studying an elliptic equation called the \textit{Babenko equation} which describes the profile of solitary waves.
Starting from the formulation of the problem as the critical point of the total energy subject to the constraint of fixed momentum, we derive the Babenko equation:
\begin{equation}
\begin{aligned}
    &(g+ c\gamma-c^2|D|)U - \sigma\left(\dfrac{U_\alpha}{J^{\frac 1 2}}\right)_\alpha+ \sigma |D|\left(\dfrac{1+|D|U}{J^{\frac 1 2}}\right)\\
   =& - \frac{\gamma^2}{2}U^2-gU|D|U -\frac{1}{2}g|D|U^2 +\frac{\gamma^2}{2}U|D|U^2 - \frac{\gamma^2}{2}U^2|D|U 
    - \frac{\gamma^2}{6}|D|U^3, 
\end{aligned}    \label{e:BabenkoEqnLR}
\end{equation}
for $U = \Im W$.
Historically,  Babenko studied periodic traveling waves in \cite{MR899856, MR898306}, but the technique of building a single elliptic equation to solve works for the solitary wave case as well.
Alternative differential and integral equations for steady waves have also been studied, such as the Nekrasov equation~\cite{Nekrasov} where the unknown is the angle the free surface makes with the horizontal rather than the height.
This alternative formulation is well-suited to the study of extreme Stokes waves, see for example~\cite{MR1231326, MR4533450}, but the Babenko equation formulation is more convenient for reducing the solitary water wave problem to simpler equations.

Working in holomorphic coordinates makes the nonlocal operators that arise in the Babenko equation simpler to work with.
In holomorphic coordinates, the Dirichlet-Neumann operator is reduced to the differential operator $|D|$, which drastically simplifies the required estimates in the proof.
For comparison, in \cite{MR4246394,MR2847283, MR3415532}, a considerable number of computations are needed to deal with the analyticity of the Dirichlet-Neumann operator and related nonlocal operators.
In the nonexistence result \cite{MR4151415}, the Babenko equation is written in holomorphic coordinates in the complex-valued form. 
Here, we consider a real-valued Babenko equation instead since it allows us to work directly with the wave profile.

Indeed, the existence of a solitary wave for \eqref{e:CVWW} is equivalent to the existence of a nontrivial solution of Babenko equation \eqref{e:BabenkoEqnLR}.
This equation arises from looking for a critical point of the energy of~\eqref{e:CVWW} subject to a constraint on the momentum, with the velocity $c$ arising as a Lagrange multiplier.
A full derivation will be given in Section \ref{s:Babenko}.

For a further comparison with the constant vorticity solitary waves result \cite{MR3415532} in finite depth, in that paper the solitary waves are studied as solutions of a constrained minimization problem with fixed momentum instead of writing down a variant of the Babenko equation.
As a result of the purely variational approach, the main ingredient of \cite{MR3415532} is the concentration compactness principle.
One needs to have a priori estimates for the energy functional in order to show that only the concentration case can happen.
On the other hand, the variational approach in \cite{MR3415532} gives the extra conditional energetic stability that is not discussed in this article.
In this article, we instead use the ideas of NLS approximation and implicit function theorem in Groves \cite{MR4246394}. 
This is a direct approximation result that gives the approximate equation.
Since the explicit solutions of stationary focusing NLS equation are well-known, this approach also tells us the approximate geometric shape of the solitary waves.

The rest of the article is organized as follows. 
In  Section \ref{s:Babenko} we use the variational approach to derive the Babenko equation \eqref{e:BabenkoEqnLR}, which is a quasilinear second-order equation that describes the profile of the solitary waves of \eqref{e:CVWW}.
In Section \ref{s:Fun}, we decompose the profile of the solitary waves into a frequency-localized part $U_1$ and a remainder part $U_2$.
We further show that $U_2$ can be expressed as an implicit function of $U_1$.
Then in Section \ref{s:Approximate}, we derive an equation \eqref{UOneReduced} for $U_1^+ = \chi^+(D)U$ and show that using a rescaling, it is equivalent to the equation \eqref{e:FocusingODE}.
In Section \ref{s:Solving}, by taking the limit as $\epsilon$ goes to zero, \eqref{e:FocusingODE} converges to a stationary focusing NLS equation, and the unique solution can be shown to exist using the implicit function theorem.
Thus, we show that \eqref{e:CVWW} has solitary wave solutions. 
Finally, in Section \ref{s:parameterRegimes}, we explicitly compute the frequencies corresponding to the critical velocities by working with a nondimensionalized version of the equation and check when the NLS model is focusing.
By an algebraic positivity argument, we guarantee two families of solitary wave solutions for all parameter regimes in the gravity-capillary case with constant vorticity. \\

\textbf{Acknowledgments.} 
The authors would like to thank Jeremy L. Marzuola for helpful discussion and Jiaqi Yang for pointing out a mistake in the initial version of this manuscript.
The authors also want to thank the anonymous reviewer for their careful reading and helpful suggestions during the revision process.
The authors were supported by the National Science Foundation under Grant No.~DMS-1928930 while they were in residence at the Simons Laufer Mathematical Sciences Institute (formerly MSRI) in Berkeley, California, USA, during the summer of 2023.
The first author was also supported by the NSF RTG DMS-2135998.
The second author is partially supported by the National Natural Science Foundation of China under Grant No.~12671283.

\section{The derivation of the Babenko equation} \label{s:Babenko}

 In this section, we derive the Babenko equation of \eqref{e:CVWW}, which describes the profile of solitary waves.
 The system \eqref{e:CVWW} is expressed in holomorphic coordinates, and both $W$ and $Q$ are holomorphic functions.
 An advantage of using holomorphic coordinates is that the Dirichlet-Neumann operator  that arises in the Zakharov-Craig-Sulem formulation of water waves reduces to the simple differential operator $|D|$.
 This helps us avoid delicate analysis of the Dirichlet-Neumann operator, and will greatly simplify our later computations. 

For a holomorphic function, its real part equals the Hilbert transform of its imaginary part.
Therefore, the entire system can be expressed  using only the imaginary parts of $(W,Q)$, and it suffices to consider just $(\Im W, \Im Q)$.
Here, instead of simply plugging  the ansatz \eqref{Form} into \eqref{e:CVWW}, we use a variational approach to derive the real-valued equation satisfied by the solitary wave profiles.

According to the computation carried out in the Appendix of \cite{MR3869381} and Section $1$ of \cite{MR3667289},
the total energy of \eqref{e:CVWW} is given by
\begin{equation}
 \begin{aligned}
    \mathcal{E} =& \frac{1}{2}\int |D|\Im Q \cdot\Im Q +g(\Im W)^2(1+ |D|\Im W) + \gamma |D|\Im Q \cdot (\Im W)^2 \nonumber\\
    &+ \frac{\gamma^2}{3}(\Im W)^3(1+|D|\Im W)  + 2\sigma\left(J^{\frac 1 2}-1-|D|\Im W\right)d\alpha, 
\end{aligned}   
\end{equation}
and the horizontal momentum is
\begin{equation*}
    \mathcal{P} = -\int |D|\Im Q \Im W + \frac{\gamma}{2}(\Im W)^2(1+ |D|\Im W)\, d\alpha , 
\end{equation*}
where the Jacobian
\begin{equation*}
    J = |1+W_\alpha|^2 = (1+|D|\Im W)^2 +(\Im W_\alpha)^2.
\end{equation*}
Here, the differential operator $|D|$ is defined by
\begin{equation*}
    |D|f(\alpha) = \partial_\alpha Hf(\alpha)  .
\end{equation*}
An alternative definition of the operator $|D|$ is given via the Fourier transform
\begin{equation*}
    \widehat{|D|f}(\xi) = |\xi|\hat{f}(\xi).
\end{equation*}
In \cite{MR688749}, it was observed by Benjamin and Olver that  the solitary waves can be characterized as a critical point of the total energy subject to the constraint of a fixed momentum.
The velocity $c$ can be viewed as the corresponding Lagrange multiplier.
In order to compute the functional derivative, we will need the following result.
\begin{lemma}
   Given a function $f(x)$ defined on $\mathbb{R}^n$, let $L\left(x, y, z\right)$ be a $C^2$ function  $\mathbb R\times \mathbb R^n\times \mathbb R\to \mathbb R$ which  is zero at the origin, along with its first derivatives.
   Consider the functional defined by
   \begin{equation*}
       \mathcal{J}(f) = \int L\left(f(x), \nabla f(x), |D|f(x)\right) dx.
   \end{equation*}
   Then the functional derivative of $\mathcal{J}(f)$ is given by
   \begin{equation*}
       \frac{\delta \mathcal{J}}{\delta f} = \frac{\partial L}{\partial f}- \operatorname{div} \left(\frac{\partial L}{\partial \nabla f}\right)+|D|\left(\frac{\partial L}{\partial |D|f}\right).
   \end{equation*}
\end{lemma}
The proof of this result is elementary; it relies on the self-adjointness properties of the differential operator $|D|$.
The details of the proof can be found in Lemma $2.1$ of \cite{rowan2023dimensional}.

Taking the functional derivatives of the energy $\mathcal{E}$ and momentum $\mathcal{P}$ with respect to $\Im Q, \Im W$, and using the results in \cite{MR688749}, we get two equations.
The first equation is for the functional derivatives with respect to $\Im Q$,
\begin{equation*}
\frac{\delta \mathcal{E}}{\delta \Im Q} = c \frac{\delta \mathcal{P}}{\delta \Im Q},
\end{equation*}
which can be simplified to
\begin{equation}
\Im Q = -\frac{\gamma}{2}(\Im W)^2 - c\Im W. \label{e:BabenkoQ}
\end{equation}

The second equation is for the functional derivative with respect to $\Im W$,
\begin{equation*}
\frac{\delta \mathcal{E}}{\delta \Im W} = c \frac{\delta \mathcal{P}}{\delta \Im W}.
\end{equation*}
By computation,
\begin{align*}
  \frac{\delta \mathcal{E}}{\delta \Im W} =& g\Im W(1+|D|\Im W) +\frac{1}{2}g|D|(\Im W)^2 +\gamma |D|\Im Q \Im W + \frac{\gamma^2}{2}(\Im W)^2(1+|D|\Im W) \\
  &+ \frac{\gamma^2}{6}|D|(\Im W)^3-\sigma\left(\dfrac{\Im W_\alpha}{J^{\frac 1 2}}\right)_\alpha+\sigma |D|\left(\dfrac{1+|D|\Im W}{J^{\frac 1 2}}\right),\\
  \frac{\delta \mathcal{P}}{\delta \Im W} =& -|D|\Im Q -\gamma \Im W(1+|D|\Im W)-\frac{\gamma}{2}|D|(\Im W)^2,
\end{align*}
which leads to the second equation
\begin{align*}
    g\Im W(1+|D|\Im W) +\frac{1}{2}g|D|(\Im W)^2 +\gamma |D|\Im Q \Im W + \frac{\gamma^2}{2}(\Im W)^2(1+|D|\Im W) + \frac{\gamma^2}{6}|D|(\Im W)^3 \\
     +c|D|\Im Q +c\gamma \Im W(1+|D|\Im W)+ c\frac{\gamma}{2}|D|(\Im W)^2 =\sigma\left(\dfrac{\Im W_\alpha}{J^{\frac 1 2}}\right)_\alpha-\sigma |D|\left(\dfrac{1+|D|\Im W}{J^{\frac 1 2}}\right). 
\end{align*}
Therefore by eliminating $\Im Q$ in the second equation using \eqref{e:BabenkoQ}, we obtain that $U: =\Im W$ solves the Babenko equation \eqref{e:BabenkoEqnLR}:
\begin{align*}
    &(g+ c\gamma-c^2|D|)U -\sigma\left(\dfrac{U_\alpha}{J^{\frac 1 2}}\right)_\alpha+\sigma |D|\left(\dfrac{1+|D|U}{J^{\frac 1 2}}\right) \nonumber\\
   =& - \frac{\gamma^2}{2}U^2-gU|D|U -\frac{1}{2}g|D|U^2 +\frac{\gamma^2}{2}U|D|U^2 - \frac{\gamma^2}{2}U^2|D|U 
    - \frac{\gamma^2}{6}|D|U^3.
\end{align*}

Note that in this derivation, $c$ is the Lagrange multiplier, while $\sigma, g$ and $\gamma$ are three fixed physical parameters.

As a consequence of \eqref{e:BabenkoQ} and the properties of holomorphic functions, we immediately obtain the following result that gives the solitary wave solutions $(W, Q)$ once we solve for a nonzero $U$ in \eqref{e:BabenkoEqnLR}.
\begin{corollary} \label{t:RecoverSolution}
If $u\neq 0$ solves \eqref{e:BabenkoEqnLR}, then \eqref{e:CVWW} has a solitary wave solution
\begin{equation*}
    \left( Hu(\alpha + ct) + iu(\alpha + ct),   -cHu(\alpha + ct)-\frac{\gamma}{2}H(u^2)(\alpha + ct)-icu(\alpha + ct)-\frac{i\gamma}{2}u^2(\alpha + ct)\right).
\end{equation*}
\end{corollary}
In the rest of this article, we will focus on finding the nonzero solution of the Babenko equation \eqref{e:BabenkoEqnLR}.

\section{Functional setting and reduction} \label{s:Fun}
The Babenko equation \eqref{e:BabenkoEqnLR} is a quasilinear elliptic equation under conditions \eqref{SignCon1} and \eqref{CoerciveCon1}.
The constrained minimization method in general cannot solve quasilinear equations like this.
Nevertheless, as pointed out in Section \ref{s:Intro}, when $c$ is slightly larger than $c_1$ or slightly smaller than $c_2$, we expect a possibility of solitary waves with small amplitude bifurcating from the linear waves with frequency either $\omega_1$ or $\omega_2$.
For these cases, we prove in this section that it suffices to consider only the part of $U$ whose frequencies are localized near $\omega_1$ or $\omega_2$.

For clarity and simplicity, we will perform the analysis in subspaces of $H^2(\mathbb{R})$, which will ultimately lead to the existence of $H^2$ solitary wave solutions.
At the end of this section, we  briefly outline the idea to generalize the analysis to $H^s(\mathbb{R})$, $s\geq 2$.

\subsection{Functional setting}
We first define the function spaces we will be using later.
These function spaces are the same as in \cite{MR4246394}.

Let $\chi(\xi)$ be the characteristic function of the set $(-|\omega|-\delta, -|\omega|+ \delta )\cup (|\omega|-\delta, |\omega|+ \delta )$ for some small positive constant $\delta$, where $\omega$ is either $\omega_1$ or $\omega_2$ depending on whether the velocity $c$ is close to $c_1$ or $c_2$.
One can decompose
\begin{equation}
    U = U_1 + U_2 : = \chi(D)U + (1-\chi(D))U. \label{UOneTwo}
\end{equation}
For our analysis below, we use the function space $\mathcal{X} = H^2(\mathbb{R})$.
It can be decomposed into the direct sum
\begin{equation*}
\mathcal{X} = \mathcal{X}_1 + \mathcal{X}_2 : = \chi(D) \mathcal{X} + (1-\chi(D)) \mathcal{X}.   
\end{equation*}
For the space $\mathcal{X}_1$, we define the associated scaled norm
\begin{equation*}
  \|u\|_{\mathfrak{E}_\omega} : =  \left(\int_{\mathbb{R}} (1+ \epsilon^{-2}(|\xi|- |\omega|)^2)|\hat{u}(\xi)|^2) d\xi \right)^{\frac{1}{2}}.
\end{equation*}
Using the fact that $U_1 \in \mathcal{X}_1$ is frequency-localized, we have the following estimates for the rescaled norm:
\begin{lemma}[\hspace{1sp}\cite{MR4246394}]
For each $u\in \mathcal{X}_1$, 
\begin{align}
&\|\hat{u} \|_{L^1} \lesssim \epsilon^{\frac{1}{2}}\|u\|_{\mathfrak{E}_\omega}, \quad \|u\|_{W^{n,\infty}} \lesssim \epsilon^{\frac{1}{2}}\|u\|_{\mathfrak{E}_\omega},\label{LOneBound} \\
&\|u\|_{H^{s_1}} \approx \|u\|_{H^{s_2}} \lesssim \|u\|_{\mathfrak{E}_\omega}, \quad s_1, s_2 \geq 0. \label{HsEkBound}
\end{align}
\end{lemma}
Due to the estimate \eqref{LOneBound}, we will work with the function space $\mathcal{Z}$ defined by
\begin{equation*}
 \mathcal{Z} = \{U \in \mathcal{S}^{'}: \| U\|_{\mathcal{Z}} := \| \hat{U}_1\|_{L^1} + \|U_2 \|_{H^2} < +\infty \}.   
\end{equation*}
Then we have the inequality
\begin{equation*}
\|U\|_{L^\infty} \lesssim \|U\|_{\mathcal{Z}}\lesssim  \epsilon^{\frac{1}{2}}\|U_1\|_{\mathfrak{E}_\omega} + \|U_2\|_{H^2}.
\end{equation*}

Let $M \lesssim \epsilon$ be a small positive number.
In the following, we will abuse the  notation and write $B_M$ for a ball of radius $M$ in $\mathcal{Z}$.
For any $U\in B_M$, $\|U\|_{W^{1,\infty}}$ is also small.

To estimate the capillary terms, we have the following approximation result:
\begin{lemma} \label{LemmaThreeTwo}
For all $U\in B_M$, the capillary terms have the following expansion:
\begin{equation}
\begin{aligned}
 &- \sigma\left(\dfrac{U_\alpha}{J^{\frac 1 2}}\right)_\alpha+ \sigma |D|\left(\dfrac{1+|D|U}{J^{\frac 1 2}}\right) =  -\sigma U_{\alpha \alpha} -\frac{\sigma}{2}|D|(|D|U)^2\\
 & +\frac{\sigma}{2}\partial_\alpha (U_\alpha)^3 +\sigma|D|(|D|U \cdot U_\alpha^2) - \sigma
 ((|D|U)^2\cdot U_\alpha)_\alpha +\tilde{R}(U),  
\end{aligned}  \label{CapillaryExpandTwo}
\end{equation}
where the remainder $\tilde{R}(U)$ consists of quartic and higher order terms in $U$ containing more than four derivatives and satisfies the estimates
\begin{equation}
\|\tilde{R}(U)\|_{L^2} \lesssim \|U\|_{\mathcal{Z}}^3\|U\|_{H^2}, \quad \|d\tilde{R}[U](V)\|_{L^2}\lesssim (\|U\|_{\mathcal{Z}}+ \|U\|_{W^{1,\infty}})^3\|V\|_{H^2}. \label{RemainderCapillaryTwo}
\end{equation}
\end{lemma}
We remark that this computation coincides with Lemma 2.10 in \cite{MR2847283}.

\begin{proof}
Observe that the two capillary terms on the left of \eqref{CapillaryExpandTwo} are the critical point of the  capillary  functional
\begin{equation*}
    I_{cap}(U) = \int \sigma\left(J^{\frac 1 2}-1-|D|U \right)d\alpha.
\end{equation*}
Recall the Taylor expansion
\begin{equation*}
    (1+x)^{\frac{1}{2}} = 1 + \frac{1}{2}x-\frac{1}{8}x^2 +\frac{1}{16}x^3-\frac{5}{128}x^4+ r(x), \quad |x|<1,
\end{equation*}
where $r(x)$ contains all quintic and higher order terms.
Since $\|U\|_{W^{1,\infty}}$ is small, one can apply the Taylor expansion up to the quartic term,
\begin{equation*}
 J^{\frac 1 2}-1-|D|U = \frac{1}{2}U_{\alpha}^2 -\frac{1}{2}|D|U (U_\alpha)^2-\frac 1 8 U_\alpha^4 +\frac{1}{2}(|D|U)^2U_\alpha^2+ r_*(U),    
\end{equation*}
where $r_*(U)$ consists of quintic and higher order terms of $U$ with more than four derivatives.

Using the identity of the Hilbert transform 
\begin{equation*}
    u^2 = (Hu)^2 -2H(uHu), \quad u\in L^2,
\end{equation*}
one can rewrite the cubic integral
\begin{equation*}
-\frac{1}{2}\int |D|U (U_\alpha)^2 \,d\alpha = \int -\frac{1}{2}(|D|U)^3 + |D|U (U_\alpha)^2 \,d\alpha = -\frac{1}{6}\int (|D|U)^3 \, d\alpha.
\end{equation*}
The functional derivative of $\int -\frac{\sigma}{6}(|D|U)^3 -\frac{\sigma}{8} U_\alpha^4\,d\alpha$
 is just $-\frac{\sigma}{2} |D|(|D|U)^2   + \frac{\sigma}{2}\partial_\alpha (U_\alpha)^3$,
 and the functional derivative of $\int \frac{\sigma}{2}(|D|U)^2U_\alpha^2\,d\alpha$ is $\sigma|D|(|D|U \cdot U_\alpha^2) - \sigma
 ((|D|U)^2\cdot U_\alpha)_\alpha$.
 The functional derivative of $\sigma \int r_* (U)\, d\alpha$ gives the remainder term $\tilde{R}(U)$.
Corresponding bounds follow from direct computation.
\end{proof}

\subsection{Reduction}
One can write the Babenko equation \eqref{e:BabenkoEqnLR} as
\begin{align*}
    &\left(g+ c_{*}\gamma-c_{*}^2|D|+ \sigma |D|^2\right)U  = (c_{*}-c)\gamma U+(c^2 -c^2_{*})|D|U +\sigma\left(\dfrac{U_\alpha}{J^{\frac 1 2}}\right)_\alpha- \sigma |D|\left(\dfrac{1+|D|U}{J^{\frac 1 2}}\right) \\
   &+ \sigma |D|^2 U  - \frac{\gamma^2}{2}U^2-gU|D|U -\frac{g}{2}|D|U^2 +\frac{\gamma^2}{2}U|D|U^2 - \frac{\gamma^2}{2}U^2|D|U 
    - \frac{\gamma^2}{6}|D|U^3,
\end{align*}
where $c_{*}$ is either $c_1$ or $c_2$ so that $|c-c_{*}| \lesssim \epsilon^2$.
The operator
\begin{equation*}
L(D): = g+ c_{*}\gamma-c_{*}^2|D|+ \sigma |D|^2
\end{equation*}
is a second order elliptic operator whose symbol
\begin{equation*}
    \ell(\xi) : = g+ c_{*}\gamma-c_{*}^2|\xi|+ \sigma |\xi|^2 \geq 0 
\end{equation*}
is nonnegative and equals zero only at $\xi = \pm \omega$.
Recall the frequency decomposition $U = U_1 +U_2= \chi(D)U + (1-\chi(D))U$.
Applying $(1-\chi(D))$ to the Babenko equation and using \eqref{UOneTwo}, \eqref{CapillaryExpandTwo}, we get $U_2 = F(U_1)+ U_3$, where $F(U_1)$ is 
\begin{equation}
F(U_1) = -(1-\chi(D))L(D)^{-1}\left[ \frac{\gamma^2}{2}U_1^2+gU_1|D|U_1 + \frac{g}{2}|D|U_1^2 -\frac{\sigma}{2}|D|(|D|U_1)^2\right], \label{FUFormula}
\end{equation}
and $U_3$ is given by
\begin{equation}
\begin{aligned}
U_3 =& (1-\chi(D))L(D)^{-1}\Big[-\frac{\sigma}{2}\partial_\alpha (U_\alpha)^3 -\sigma|D|(|D|U \cdot U_\alpha^2) + \sigma
 ((|D|U)^2\cdot U_\alpha)_\alpha 
\\
    &+\frac{\gamma^2}{2}U|D|U^2- \frac{\gamma^2}{2}U^2|D|U - \frac{\gamma^2}{6}|D|U^3 - \tilde{R}(U) \Big]+ F(U)-F(U_1).
\end{aligned}    \label{UTwoImplicit}
\end{equation}
Here, $\tilde{R}(U)$ is the last term of \eqref{CapillaryExpandTwo}.

Note that
\begin{equation*}
    f \mapsto \mathcal{F}^{-1} \left[\dfrac{1-\chi(\xi)}{\ell(\xi)}\hat{f}(\xi) \right]
\end{equation*}
is a bounded linear operator from $L^2(\mathbb{R})$ to $\mathcal{X}_2$.

In the following, we show that $U_3$ can be expressed implicitly as a function of $U_1$, so that $U$ is an implicit function of $U_1$.
The key lemma of the proof is the following implicit function theorem in Banach spaces.
\begin{lemma}[Implicit function theorem  \cite{MR4246394}] \label{t:Implicit}
Let $\mathcal{E}_1$, $\mathcal{E}_2$ be two  Banach spaces,  $X_1, X_2$ be closed convex sets in, respectively, $\mathcal{E}_1$, $\mathcal{E}_2$ containing the origin and $\mathcal{G}: X_1\times X_2 \rightarrow \mathcal{E}_2$ be a smooth function.
Suppose that there exists a continuous function $r: X_1 \rightarrow [0,\infty)$, such that 
\begin{equation*}
  \|\mathcal{G}(x_1, 0) \| \leq \frac{r}{2}, \quad \|d_2 \mathcal{G}[x_1, x_2] \| \leq \frac{1}{3} 
\end{equation*}
for each $x_2\in \bar{B}_r(0) \subset X_2$ and each $x_1\in X_1$.

Under these hypotheses there exists for each $x_1 \in X_1$ a unique solution $x_2 = x_2(x_1)$ of the fixed-point equation $x_2 = \mathcal{G}(x_1, x_2)$ satisfying $x_2(x_1)\in \bar{B}_r(0)$.
Moreover $x_2(x_1)$ is a smooth function of $x_1 \in X_1$, and in particular satisfies the estimate
\begin{equation*}
\|d x_2[x_1]\| \leq 2 \|d_1 \mathcal{G}[x_1, x_2(x_1)] \|.
\end{equation*}
\end{lemma}

To adapt the above implicit function theorem to the equation \eqref{UTwoImplicit}, we choose the function spaces $\mathcal{E}_1 = \mathcal{X}_1 = \chi(D)H^2$ and $\mathcal{E}_2 =\mathcal{X}_2 = (1-\chi(D))H^2$ and the sets
\begin{equation*}
 X_1 = \{U_1\in \mathcal{X}_1: \| U_1\|_{\mathfrak{E}_\omega}\leq R_1 \}, \quad X_2 = \{U_3\in \mathcal{X}_2: \| U_3\|_{H^2}\leq R_2 \}. 
\end{equation*}
For the given small ball $B_M$ in $\mathcal{Z}$, by \eqref{LOneBound}, we pick $R_1<\epsilon^{-\frac{1}{2}}M \lesssim \epsilon^{\frac{1}{2}} \ll 1$ such that $\| \hat{U}_1\|_{L^1}< \frac{M}{2}$.
The value of $R_2$ is chosen such that $\|F(U_1)+U_3\|_{H^2}<\frac{M}{2}$ for all $U_1\in X_1$ and $U_3\in X_2$.
As a consequence, $U = U_1+U_2 \in B_M$.

Let $x_1 = U_1$, $x_2 = U_3$, and $\mathcal{G}(U_1, U_3)$ be the right-hand side of \eqref{UTwoImplicit}.
We further write $\mathcal{G}(U_1, U_3) = \mathcal{G}_1(U_1, U_3) + \mathcal{G}_2(U_1, U_3) + \mathcal{G}_3(U_1, U_3)$, where $U =U_1+ U_2$ and
\begin{align*}
 \mathcal{G}_1(U_1, U_3) =& -(1-\chi(D))L(D)^{-1}\tilde{R}(U), \\
  \mathcal{G}_2(U_1, U_3) =& (1-\chi(D))L(D)^{-1}\Big[  -\frac{\sigma}{2}\partial_\alpha (U_\alpha)^3 -\sigma|D|(|D|U \cdot U_\alpha^2) + \sigma
 ((|D|U)^2\cdot U_\alpha)_\alpha\\
&+\frac{\gamma^2}{2}U|D|U^2 - \frac{\gamma^2}{2}U^2|D|U - \frac{\gamma^2}{6}|D|U^3\Big], \\
     \mathcal{G}_3(U_1, U_3) =&  F(U)-F(U_1). 
 \end{align*}

In the following, we estimate $\mathcal{G}_i(U_1, U_3)$, $d_1\mathcal{G}_i(U_1, U_3)$ and $d_2\mathcal{G}_i(U_1, U_3)$ successively for $i = 1,2,3$ in order to apply Lemma \ref{t:Implicit}.
Before that, we first give an estimate for $F(U_1)$.
\begin{lemma} 
For each $U_1\in X_1$, 
\begin{equation}
\|F(U_1) \|_{H^2} \lesssim \epsilon^{\frac{1}{2}} \|U_1\|_{\mathfrak{E}_\omega}^2, \quad \|dF[U_1] \|_{\mathcal{L}(\mathcal{X}_1, \mathcal{X}_2)} \lesssim \epsilon^{\frac{1}{2}} \|U_1\|_{\mathfrak{E}_\omega}. \label{FUOneEstimate}
\end{equation}
As a consequence of $R_1 \ll 1$, $\|U_2\|_{H^2}\lesssim \epsilon^{\frac{1}{2}} \|U_1\|_{\mathfrak{E}_\omega}+ \|U_3 \|_{H^2}$.
\end{lemma}
\begin{proof}
It follows that for the $H^2$ norm of $F(U_1)$,
\begin{equation*}
\| F(U_1)\|_{H^2} \lesssim \|U_1^2 + U_1|D|U_1 \|_{L^2} + \|U_1^2\|_{\dot{H}^1} + \|(|D|U_1)^2\|_{H^1} \lesssim \|U_1 \|_{W^{2,\infty}}\|U_1 \|_{H^1}\lesssim \epsilon^{\frac{1}{2}} \|U_1\|_{\mathfrak{E}_\omega}^2.
\end{equation*} 
Taking the derivative,
\begin{equation*}
dF[U_1](V) = (\chi(D)-1)L(D)^{-1}\left[ \gamma^2U_1 V +gV|D|U_1+gU_1|D|V + g|D|(U_1 V) -\sigma|D|(|D|U_1 |D|V)\right]. 
\end{equation*}
Hence, for any $V\in \mathcal{X}_1$,
\begin{equation*}
\|dF[U_1]V \|_{H^2} \lesssim \| U_1\|_{W^{2,\infty}}\|V\|_{H^2} \lesssim \epsilon^{\frac{1}{2}} \|U_1\|_{\mathfrak{E}_\omega}\|V\|_{H^2},
\end{equation*}
which gives the second estimate of \eqref{FUOneEstimate}.
\end{proof}

\begin{lemma} \label{t:GOneEstimate}
For each $U_1\in X_1$ and $U_3\in X_2$,
\begin{align*}
&\|\mathcal{G}_1(U_1, U_3) \|_{H^2} \lesssim (\epsilon^{\frac{1}{2}}\|U_1\|_{\mathfrak{E}_\omega}+ \|U_3\|_{H^2})^3 (\|U_1\|_{\mathfrak{E}_\omega}+ \|U_3\|_{H^2}), \\
&\|d_1\mathcal{G}_1(U_1, U_3) \|_{\mathcal{L}(\mathcal{X}_1, \mathcal{X}_2)}\lesssim (\epsilon^{\frac{1}{2}}\|U_1\|_{\mathfrak{E}_\omega}+ \|U_3\|_{H^2})^3(1+ \epsilon^{\frac{1}{2}}\|U_1\|_{\mathfrak{E}_\omega}), \\
&\|d_2\mathcal{G}_1(U_1, U_3) \|_{\mathcal{L}(\mathcal{X}_2, \mathcal{X}_2)}\lesssim (\epsilon^{\frac{1}{2}}\|U_1\|_{\mathfrak{E}_\omega}+ \|U_3\|_{H^2})^3. 
\end{align*}
\end{lemma}
\begin{proof}
  We use \eqref{RemainderCapillaryTwo} to get 
\begin{equation*}
     \| (1-\chi(D))L(D)^{-1} \tilde{R}(U)\|_{H^2} \lesssim \|U\|_{\mathcal{Z}}^3\|U\|_{H^2}\lesssim (\epsilon^{\frac{1}{2}}\|U_1\|_{\mathfrak{E}_\omega}+ \|U_3\|_{H^2})^3 (\|U_1\|_{\mathfrak{E}_\omega}+ \|U_3\|_{H^2}).
\end{equation*}
The estimates for the derivative follow from the derivative estimate \eqref{RemainderCapillaryTwo}.
\end{proof}

\begin{lemma} \label{t:GThreeEstimate}
For each $U_1\in X_1$ and $U_3\in X_2$,
\begin{align*}
 &\|\mathcal{G}_2(U_1, U_3) \|_{H^2} \lesssim (\epsilon^{\frac{1}{2}}  \|U_1\|_{\mathfrak{E}_\omega} +\|U_3\|_{H^2})^2(\|U_1\|_{\mathfrak{E}_\omega}+\|U_3\|_{H^2}), \\ 
 &\|d_2\mathcal{G}_2(U_1, U_3) \|_{\mathcal{L}(\mathcal{X}_2, \mathcal{X}_2)}\lesssim (\epsilon^{\frac{1}{2}}  \|U_1\|_{\mathfrak{E}_\omega} +\|U_3\|_{H^2})^2,\\   
 &\|d_1\mathcal{G}_2(U_1, U_3) \|_{\mathcal{L}(\mathcal{X}_1, \mathcal{X}_2)}\lesssim (\epsilon^{\frac{1}{2}}  \|U_1\|_{\mathfrak{E}_\omega} +\|U_3\|_{H^2})(  \|U_1\|_{\mathfrak{E}_\omega} +\|U_3\|_{H^2}).
\end{align*}
\end{lemma}
\begin{proof}
Using the fact that $(1-\chi(D))L(D)^{-1}$ is bounded from  $L^2$ to $H^2$ and \eqref{FUOneEstimate}, we compute:
\begin{align*}
 &\|(1-\chi(D))L(D)^{-1}\partial_\alpha (U_\alpha)^3 \|_{H^2} 
   \lesssim (\epsilon^{\frac{1}{2}}  \|U_1\|_{\mathfrak{E}_\omega} +\|U_3\|_{H^2})^2(\|U_1\|_{\mathfrak{E}_\omega}+\|U_3\|_{H^2}),\\
&\|(1-\chi(D))L(D)^{-1}|D|(|D|U \cdot U_\alpha^2) \|_{H^2} 
   \lesssim (\epsilon^{\frac{1}{2}}  \|U_1\|_{\mathfrak{E}_\omega} +\|U_3\|_{H^2})^2(\|U_1\|_{\mathfrak{E}_\omega}+\|U_3\|_{H^2}),\\
&\|(1-\chi(D))L(D)^{-1}((|D|U)^2\cdot U_\alpha)_\alpha \|_{H^2} 
   \lesssim (\epsilon^{\frac{1}{2}}  \|U_1\|_{\mathfrak{E}_\omega} +\|U_3\|_{H^2})^2(\|U_1\|_{\mathfrak{E}_\omega}+\|U_3\|_{H^2}),\\
 & \|(1-\chi(D))L(D)^{-1}(U|D|U^2-U^2|D|U) \|_{H^2} \lesssim \|U|D|U^2 \|_{L^2} + \|U^2 |D|U \|_{L^2}\\
 \lesssim & \|U\|_{L^\infty}^2 \|U_\alpha\|_{L^2} \lesssim (\epsilon^{\frac{1}{2}}  \|U_1\|_{\mathfrak{E}_\omega} +\|U_3\|_{H^2})^2(\|U_1\|_{\mathfrak{E}_\omega}+\|U_3\|_{H^2}),\\
 & \|(1-\chi(D))L(D)^{-1}(|D|U^3) \|_{H^2} \lesssim \||D|U^3\|_{L^2}\lesssim \|U^2U_\alpha\|_{L^2}\\
 \lesssim & \|U\|_{L^\infty}^2 \|U_\alpha\|_{L^2} \lesssim (\epsilon^{\frac{1}{2}}  \|U_1\|_{\mathfrak{E}_\omega} +\|U_3\|_{H^2})^2(\|U_1\|_{\mathfrak{E}_\omega}+\|U_3\|_{H^2}).
\end{align*}
Combining these estimates, we obtain the first bound.

As for the expression for the operator $d_2\mathcal{G}_2(U_1, U_3)$, we apply the chain rule. 
\begin{align*}
&d_2\mathcal{G}_2(U_1, U_3)V =  (1-\chi(D))L(D)^{-1}\Big[ -\frac{3}{2}\partial_\alpha(U_\alpha^2V_\alpha) -\sigma|D|(|D|V\cdot U_\alpha^2) \\
-& 2\sigma|D|(|D|U\cdot U_\alpha\cdot V_\alpha)+ 2\sigma(|D|U\cdot |D|V\cdot U_\alpha)_\alpha + \sigma((|D|U)^2\cdot V_\alpha)_\alpha
+ \frac{\gamma^2}{2}V|D|U^2 
\\
+& \gamma^2 U|D|(UV) - \gamma^2UV|D|U - \frac{\gamma^2}{2}U^2|D|V - \frac{\gamma^2}{2}|D|(U^2V) \Big],
\end{align*}
where $U = U_1 +F(U_1)+U_3$, for each $V\in \mathcal{X}_1$.
Hence,
\begin{align*}
& \|\partial_\alpha(U_\alpha^2V_\alpha) \|_{L^2} \lesssim (\epsilon^{\frac{1}{2}}  \|U_1\|_{\mathfrak{E}_\omega} +\|U_2\|_{H^2})^2 \|V\|_{H^2},\\
& \| |D|(|D|V\cdot U_\alpha^2) + 2|D|(|D|U\cdot U_\alpha\cdot V_\alpha) + 2(|D|U\cdot |D|V\cdot U_\alpha)_\alpha + ((|D|U)^2\cdot V_\alpha)_\alpha \|_{L^2}\\
\lesssim& (\epsilon^{\frac{1}{2}}  \|U_1\|_{\mathfrak{E}_\omega} +\|U_2\|_{H^2})^2 \|V\|_{H^2},\\
& \|(V|D|U^2+2U|D|(UV)-2UV|D|U +U^2|D|V+|D|(U^2V))\|_{L^2}\\
\lesssim& (\epsilon^{\frac{1}{2}}  \|U_1\|_{\mathfrak{E}_\omega} +\|U_2\|_{H^2})^2 \|V\|_{H^2}. 
\end{align*}
The estimate for $d_1\mathcal{G}_2(U_1, U_3)$ is similar to that of $d_2\mathcal{G}_2(U_1, U_3)$.
\end{proof}

\begin{lemma} \label{t:GTwoEstimate}
For each $U_1\in X_1$ and $U_2\in X_2$, $\mathcal{G}_3(U_1, U_3) = F(U)-F(U_1)$ satisfies the estimate
\begin{align*}
&\|\mathcal{G}_3(U_1, U_3) \|_{H^2} \lesssim \epsilon \|U_1 \|^3_{\mathfrak{E}_\omega} + \epsilon^{\frac{1}{2}}  \|U_1 \|^2_{\mathfrak{E}_\omega}\|U_3\|_{H^2}+ \epsilon^{\frac{1}{2}}  \|U_1 \|_{\mathfrak{E}_\omega}\|U_3\|_{H^2} + \|U_3\|^2_{H^2}, \\
&\|d_2 \mathcal{G}_3(U_1, U_3)\|_{\mathcal{L}^2(\mathcal{X}_2, \mathcal{X}_2)} \lesssim \epsilon^{\frac{1}{2}}\|U_1 \|_{\mathfrak{E}_\omega} + \|U_3\|_{H^2},\\
&\|d_1 \mathcal{G}_3(U_1, U_3)\|_{\mathcal{L}^2(\mathcal{X}_1, \mathcal{X}_2)} \lesssim \epsilon \| U_1\|^2_{\mathfrak{E}_\omega} +\epsilon^{\frac{1}{2}}\|U_1 \|_{\mathfrak{E}_\omega}\|U_3\|_{H^2} + \epsilon^{\frac{1}{2}}\|U_3\|_{H^2}.
\end{align*}
\end{lemma}

\begin{proof}
Each term of 
\begin{equation*}
 \frac{\gamma^2}{2}U^2+gU|D|U + \frac{1}{2}g|D|U^2 -\frac{\sigma}{2}|D|(|D|U)^2
\end{equation*}
in the expression of $F(U)$ is a quadratic term of $U$ with at most three derivatives.
Hence, for each term in $\mathcal{G}_2(U_1, U_3)$, it can be expressed by
\begin{equation*}
    (1-\chi(D))L(D)^{-1}[2m(U_1, F(U_1) + U_3) +m(F(U_1) + U_3, F(U_1) + U_3)],
\end{equation*}
where $m(u,v)$ is a bilinear term with at most three derivatives that satisfies
\begin{equation*}
\|m(u,v)\|_{L^2} \lesssim (\|u\|_{\mathcal{Z}}+ \| u\|_{W^{1,\infty}})\|v \|_{H^2}, \quad \forall u,v\in H^2(\mathbb{R}).
\end{equation*}
Then the estimates follow directly from \eqref{LOneBound} and \eqref{FUOneEstimate}.
\end{proof}

Putting the results in Lemma \ref{t:GOneEstimate},  Lemma \ref{t:GThreeEstimate} and Lemma \ref{t:GTwoEstimate} together, for each $U_1\in X_1$, we choose $M$  and $\epsilon$ small enough.
Let $r(U_1) = C\epsilon\| U_1\|^2_{\mathfrak{E}_\omega}$ for some large constant $C$; then for $U_3 \in \overline{B}_{r(U_1)}(0)\subseteq B_{\frac{M}{2}}(0)\subseteq X_2$,
\begin{equation*}
\|\mathcal{G}(U_1, 0)\|_{H^2} \lesssim \epsilon  \|U_1\|_{\mathfrak{E}_\omega}^3 \lesssim
 \frac{1}{2}r(U_1).
\end{equation*}
Moreover, 
\begin{align*}
&\|d_2 \mathcal{G}(U_1, U_3)\|_{\mathcal{L}(\mathcal{X}_2, \mathcal{X}_2)} \lesssim  (\epsilon^{\frac{1}{2}}  \|U_1\|_{\mathfrak{E}_\omega} +\|U_3\|_{H^2})(1+\epsilon^{\frac{1}{2}}  \|U_1\|_{\mathfrak{E}_\omega} +\|U_3\|_{H^2}) \ll\frac{1}{3}, \\
&\|d_1 \mathcal{G}(U_1, U_3)\|_{\mathcal{L}(\mathcal{X}_1, \mathcal{X}_2)} \lesssim   (\epsilon^{\frac{1}{2}}  \|U_1\|_{\mathfrak{E}_\omega} +\|U_3\|_{H^2})(\|U_1\|_{\mathfrak{E}_\omega} +\|U_3\|_{H^2}).
\end{align*}
It then follows from Lemma \ref{t:Implicit} and the assumption  $\|U_1\|_{\mathfrak{E}_\omega} \leq \epsilon^{\frac{1}{2}}$ that $U_3$ can be expressed implicitly as a function of $U_1$:
\begin{proposition} \label{t:UTwoBound}
The equation \eqref{UTwoImplicit} has a unique solution $U_3 \in X_2$ which depends smoothly upon $U_1 \in X_1$.
In addition,
\begin{equation*}
\|U_3(U_1)\|_{H^2} \lesssim \epsilon\|U_1\|^2_{\mathfrak{E}_\omega}, \quad  \|dU_3(U_1)\|_{\mathcal{L}(\mathcal{X}_1, \mathcal{X}_2)} \leq \epsilon\|U_1\|_{\mathfrak{E}_\omega}.
\end{equation*}
This means that $U_3$, and thus $U$ itself, is an implicit function of $U_1$.
\end{proposition}

To further show that $U_2$ is in any $H^s(\mathbb{R})$, we just need to repeat the above analysis and work in  the function space $\mathcal{X}^s = H^s(\mathbb{R})$, $s\geq 2$.
Again it is decomposed into the direct sum
\begin{equation*}
\mathcal{X}^s = \mathcal{X}^s_1 + \mathcal{X}^s_2 : = \chi(D) \mathcal{X}^s + (1-\chi(D)) \mathcal{X}^s. 
\end{equation*}
Note that all $\chi_0(\epsilon D)H^s(\mathbb{R}), s\geq 0$ are  topologically equivalent.
The auxiliary function space $\mathcal{Z}^s$ is defined by
\begin{equation*}
 \mathcal{Z}^s = \{U \in \mathcal{S}^{'}: \| U\|_{\mathcal{Z}} := \| \hat{U}_1\|_{L^1} + \|U_2 \|_{H^s} < +\infty \},   
\end{equation*}
so that
\begin{equation*}
\|U\|_{L^\infty} \lesssim \|U\|_{\mathcal{Z}}\lesssim \|U_1\|_{\mathfrak{E}_\omega} + \|U_2\|_{H^s}.
\end{equation*}

In the proof of the above auxiliary lemmas, one of the key inequalities is
\begin{equation*}
\left\|\prod_{i=1}^n f_i \right\|_{L^2} \leq \|f_j\|_{L^2} \prod_{i\neq j} \| f_i\|_{L^\infty}. 
\end{equation*}
This inequality is no longer true when $L^2$ is replaced by $H^{s}$, $s>0$. 
Instead, we have
\begin{equation}
\left\|\prod_{i=1}^n f_i \right\|_{H^s} \lesssim \sum_{j=1}^n \left(\|f_j\|_{H^s} \prod_{i\neq j} \| f_i\|_{L^\infty}\right), \quad s>0. \label{ProductHs}
\end{equation}
When $n=2$, inequality \eqref{ProductHs} is just Corollary 2.86 in \cite{MR2768550}.
When $n\geq 3$, it can be easily proved by induction from the base case $n\geq 2$.

Other than the inequality \eqref{ProductHs}, the rest of the analysis is similar to the analysis above.
By using the implicit function theorem Lemma \ref{t:Implicit}, we get the following qualitative result.
\begin{corollary} \label{t:UTwoHsExist}
For $U_1 \in X_1^s$, which is a sufficiently small ball in the space $\mathcal{X}_1^s$, $U_2$ exists and is in $\mathcal{X}_2^s \subset H^s(\mathbb{R})$.
\end{corollary}

\section{Derivation of the reduced equation} \label{s:Approximate}
From Proposition \ref{t:UTwoBound}, $U_2$ is a function of $U_1$ when $\|U_1\|_{\mathfrak{E}_\omega} \ll 1$.
One can further decompose $U_1$ as the sum of positive and negative frequency pieces,
\begin{equation*}
U_1 = U_1^+ + U_1^- : = \chi^{+}(D)U + \chi^{-}(D)U,
\end{equation*}
where $U_1^+ = \overline{U_1^-}$, and the frequencies of $U_1^\pm$ are supported on $(\pm |\omega|-\delta, \pm |\omega|+ \delta)$.  
Hence, to solve the Babenko equation \eqref{e:BabenkoEqnLR}, it suffices to apply the multiplier $\chi^{+}(D)$ and solve the frequency-localized equation for $U_1^+$:
\begin{equation}
\begin{aligned}
    &\left(g+ c_{*}\gamma-c_{*}^2|D|+ \sigma |D|^2\right)U_1^+  + (c-c_{*})\gamma U_1^+ +(c^2_{*}-c^2)|D|U_1^+ \\
    = &\chi^{+}(D)\Big( \frac{\sigma}{2}|D|(|D|U)^2 - \frac{\gamma^2}{2}U^2
   -gU|D|U -\frac{1}{2}g|D|U^2- \frac{\sigma}{2}\partial_\alpha (U_\alpha)^3+\sigma \left((|D|U)^2U_\alpha\right)_\alpha \\
   &- \sigma |D|\left(|D|UU_\alpha^2\right)
     +\frac{\gamma^2}{2}U|D|U^2 - \frac{\gamma^2}{2}U^2|D|U 
    - \frac{\gamma^2}{6}|D|U^3 - \tilde{R}(U)\Big),
\end{aligned}  \label{UOneReduced}
\end{equation}
where $U = U_1 + U_2(U_1)$ and $\tilde{R}(U)$ is the remainder term given in \eqref{CapillaryExpandTwo}.
In this section,  we simplify the equation \eqref{UOneReduced} and derive a reduced equation.

To give a qualitative estimate for the higher-order remainder term, we use the notations in \cite{MR4246394}.
\begin{definition}[Symbols for higher order terms]
\begin{enumerate}
\item  The symbol $\mathcal{O}(\epsilon^a \| U_1\|^b_{\mathfrak{E}_\omega})$  denotes the class of smooth functions $\mathcal{R}: X_1 \rightarrow L^2(\mathbb{R})$  which satisfy the estimates
\begin{equation*}
\|\mathcal{R}(U_1) \|_{L^2} \lesssim \epsilon^a \| U_1\|^b_{\mathfrak{E}_\omega}, \quad \|d \mathcal{R}[U_1] \|_{\mathcal{L}(\mathcal{X}_1, L^2)} \lesssim \epsilon^a \| U_1\|^{b-1}_{\mathfrak{E}_\omega},
\end{equation*}
for each $U_1 \in X_1$ (where $a\geq 0$, $b\geq 0$).
The underscored notation $\underline{\mathcal{O}}_{+}(\epsilon^a \| U_1\|^b_{\mathfrak{E}_\omega})$  indicates additionally that the Fourier transform of $\mathcal{R}(U_1)$  lies in $\chi^{+}(D)$:
\begin{equation*}
\underline{\mathcal{O}}_{+}(\epsilon^a \| U_1\|^b_{\mathfrak{E}_\omega}) := \chi^{+}(D)\mathcal{O}(\epsilon^a \| U_1\|^b_{\mathfrak{E}_\omega}), 
\end{equation*}
where $\chi^{+}$ is the  characteristic function of the set $(|\omega|-\delta, |\omega|+\delta)$.
Similarly, for the symbol class $\underline{\mathcal{O}}(\epsilon^a \| U_1\|^b_{\mathfrak{E}_\omega})$, we just need to replace $\chi^+ (D)$ by $\chi(D)$ in the above definition.

\item  The symbol $\underline{\mathcal{O}}^\epsilon_n (\|u\|^a_{H^1})$  denotes $\chi_0(\epsilon D)\mathcal{R}(u)$, where $\mathcal{R}$  is a smooth function $B_R(0) \subset \chi_0(\epsilon D)H^1(\mathbb{R})\rightarrow H^n(\mathbb{R})$  which satisfies the estimates
\begin{equation*}
\|\mathcal{R}(u)\|_{H^n} \lesssim \|u\|^a_{H^1}, \quad \|d\mathcal{R}(u) \|_{\mathcal{L}(H^1, H^n)} \lesssim \|u\|^{a-1}_{H^1},   
\end{equation*}
for each $u \in B_R(0)$ with $a\geq 1$, $n\geq 0$.
Here $\chi_0$ is the characteristic function of $(-\delta, \delta)$.
\end{enumerate}   
\end{definition}

Using the definitions above, one can show the following  estimates.
\begin{lemma} \label{t:ApproximateIdentity}
We have the following identities for each $U_1 \in X_1$:
\begin{enumerate}
\item $|D|U_1^{+} = |\omega|U_1^{+} +\underline{\mathcal{O}}_{+}(\epsilon \| U_1\|_{\mathfrak{E}_\omega})$.\label{est:modDu1Plus}
\item $|D|(U_1^{+})^2 = 2|\omega|(U_1^{+})^2+\underline{\mathcal{O}}_{+}(\epsilon^{\frac{3}{2}} \| U_1\|^2_{\mathfrak{E}_\omega})$.\label{est:modDu1PlusSquared}
\item $|D|(U_1^{+}U_1^{-}) = \underline{\mathcal{O}}(\epsilon^{\frac{3}{2}} \| U_1\|^2_{\mathfrak{E}_\omega})$.\label{est:modDu1PlusU1Minus} 
\item $U^{+}_1 |D|U^{+}_1 = |\omega|(U_1^{+})^2+\underline{\mathcal{O}}_{+}(\epsilon^{\frac{3}{2}} \| U_1\|^2_{\mathfrak{E}_\omega})$.\label{est:u1PlusModDu1Plus}
\item $|D|(|D|U^+_1 |D|U^+_1) = 2|\omega|^3(U_1^{+})^2+\underline{\mathcal{O}}_+(\epsilon^{\frac{3}{2}} \| U_1\|^2_{\mathfrak{E}_\omega})$.\label{est:modDofModDproductPlus}
\item $|D|(|D|U^+_1 |D|U^-_1) = \underline{\mathcal{O}}(\epsilon^{\frac{3}{2}} \| U_1\|^2_{\mathfrak{E}_\omega})$.\label{est:modDofModDproductMinus}
\item $|D|((U_1^+)^2 U_1^-) = |\omega|(U_1^+)^2 U_1^- +\underline{\mathcal{O}}_{+}(\epsilon^{2} \| U_1\|^3_{\mathfrak{E}_\omega})$. \label{est:modDCubic}
\item $U_1^-|D|(U_1^+)^2 = 2|\omega|(U_1^+)^2 U_1^- + \underline{\mathcal{O}}_{+}(\epsilon^{2} \| U_1\|^3_{\mathfrak{E}_\omega})$. \label{est:eight}
\item $U_1^+ U_1^- |D|U_1^+ = |\omega|(U_1^+)^2 U_1^- +\underline{\mathcal{O}}_{+}(\epsilon^{2} \| U_1\|^3_{\mathfrak{E}_\omega})$. \label{est:night}
\item $U_1^+ U_1^+ |D|U_1^- = |\omega|(U_1^+)^2 U_1^- +\underline{\mathcal{O}}_{+}(\epsilon^{2} \| U_1\|^3_{\mathfrak{E}_\omega})$. \label{est:ten}
\item $U_1^{+}|D|(U_1^{+}U_1^{-}) = \underline{\mathcal{O}}_{+}(\epsilon^{2} \| U_1\|^3_{\mathfrak{E}_\omega})$.\label{est:XI} 
\item $\partial_\alpha U_1^+ = i|\omega|U_1^+ +\underline{\mathcal{O}}_{+}(\epsilon \| U_1\|_{\mathfrak{E}_\omega})$.\label{est:XII} 
\item $\partial_\alpha(U_{1,\alpha}^{+,2} U_{1,\alpha}^{-}) = -|\omega|^4 (U_1^+)U_1^- +\underline{\mathcal{O}}_{+}(\epsilon^{2} \| U_1\|^3_{\mathfrak{E}_\omega})$. \label{est:XIII}
\item $B_1(L^{-1}B_2(U_1^+, U_1^+), U_1^-) = \ell^{-1}(2|\omega|)B_1(B_2(U_1^+, U_1^+), U_1^-)  +\underline{\mathcal{O}}_{+}(\epsilon^{2} \| U_1\|^3_{\mathfrak{E}_\omega})$, for any bilinear forms $B_1(U_1, U_2), B_2(U_1, U_2)$.
\item $B_1(L^{-1}B_2(U_1^+, U_1^-), U_1^+) = \ell^{-1}(0)B_1(B_2(U_1^+, U_1^-), U_1^+)  +\underline{\mathcal{O}}_{+}(\epsilon^{2} \| U_1\|^3_{\mathfrak{E}_\omega})$, for any bilinear forms $B_1(U_1, U_2), B_2(U_1, U_2)$.
\end{enumerate}
\end{lemma}
We remark that the proof of this lemma is the same as Lemma $4.5$ in \cite{MR4246394}, and is thus omitted. 
Also see Proposition $4.13$ in \cite{MR2847283} and Lemma $4.23$ in \cite{MR3415532} for similar results.

For the last two terms on the left-hand side of \eqref{UOneReduced}, we have the following result.
\begin{lemma}
Let $U_1\in X_1$; then
\begin{equation}
\begin{aligned}
(c-c_{*})\gamma U_1^+ +(c^2_{*}-c^2)|D|U_1^+ &= \pm \epsilon^2 (\gamma -2c_{*}|\omega|) U_1^+  + \underline{\mathcal{O}}_{+}(\epsilon^4 \| U_1\|_{\mathfrak{E}_\omega}) \\
& = \epsilon ^2 \sqrt{\gamma^2-4\omega(g+\sigma \omega^2)}U_1^+  + \underline{\mathcal{O}}_{+}(\epsilon^4 \| U_1\|_{\mathfrak{E}_\omega}),
\end{aligned}
\label{ApproximateOne}
\end{equation}
where the choice of $\pm$ depends on whether $c_*$ is $c_1$ or $c_2$.
\end{lemma}
\begin{proof}
Recall that $c-c_* = \pm \epsilon^2$, so that
\begin{equation*}
c^2-c^2_* = (c_* \pm \epsilon^2)^2 -c^2_* =\pm 2\epsilon^2 c_* + \epsilon^4.
\end{equation*}
 Then the first equality of \eqref{ApproximateOne} follows directly from $(1)$ of Lemma \ref{t:ApproximateIdentity}.
By \eqref{Defcpm}, 
 \begin{equation*}
    c_{*}(\omega) = \dfrac{-\gamma \pm \sqrt{\gamma^2 -4\omega(g+\sigma \omega^2)}}{2\omega}. 
\end{equation*}
Since $\omega<0$,
\begin{equation*}
 \gamma -2c_{*}|\omega| = \pm \sqrt{\gamma^2-4\omega(g+\sigma \omega^2)},  
\end{equation*}
which gives the second equality of \eqref{ApproximateOne}.
\end{proof}

We now compute the leading term of the right-hand side of \eqref{UOneReduced} using Lemma \ref{t:ApproximateIdentity}.
For the quadratic terms, they become cubic after applying the frequency projection $\chi^+(D)$.
\begin{lemma}
Let $U_1\in X_1$, $U_2\in X_2$ and $U= U_1 + U_2$; then
\begin{equation}
\begin{aligned}
 &\chi^{+}(D)\left( \frac{\sigma}{2}|D|(|D|U)^2 - \frac{\gamma^2}{2}U^2
   -gU|D|U -\frac{g}{2}|D|U^2 \right)    \\
   =& a_3 \chi^+(D)\left(|U_1^+|^2U_1^+\right) +\underline{\mathcal{O}}_{+}(\epsilon^{2} \| U_1\|^3_{\mathfrak{E}_\omega}),   
\end{aligned}
 \label{ApproximateTwo}
\end{equation}
where $a_3$ is a constant given by
\begin{equation} \label{aThreeDef}
    a_3 = (g+c_* \gamma)^{-1}\left(2\sigma^2 |\omega|^6 -8g\sigma |\omega|^4-2\sigma \gamma^2|\omega|^3 + 12g^2 |\omega|^2 +8g\gamma^2 |\omega| + \frac{3}{2}\gamma^4 \right).
\end{equation}
\end{lemma}

\begin{proof}
Each of these terms on the left of \eqref{ApproximateTwo} is of $\chi^+(D)\mathcal{Q}(U,U)$ type, where $\mathcal{Q}(U,U)$ is a bilinear term in $U$ with at most three derivatives.

We first consider $\chi^+(D)\mathcal{Q}(U_1, U_1)$.
Since the Fourier transforms of $\mathcal{Q}(U^+_1, U^+_1)$, $\mathcal{Q}(U^+_1, U^-_1)$, $\mathcal{Q}(U^-_1, U^-_1)$ are supported on $(2|\omega|-2\delta, 2|\omega|+2\delta)$, $(-2\delta, +2\delta)$, $(-2|\omega|-2\delta, -2|\omega|+2\delta)$ respectively, these frequency supports are disjoint from the set $(|\omega|-\delta, |\omega|+\delta)$.
Hence, $\chi^+(D)\mathcal{Q}(U_1, U_1) = 0$. 

Next, we estimate  $\chi^+(D)\mathcal{Q}(U_2, U_2)$,
\begin{equation*}
\|\chi^+(D)\mathcal{Q}(U_2, U_2)\|_{H^2}\lesssim \epsilon^{2}\|U_1\|^4_{\mathfrak{E}_\omega}, \quad
 \|\chi^+(D)d\mathcal{Q}(U_2, U_2)\|_{\mathcal{L}(\mathcal{X}_1, L^2)}\lesssim \epsilon^{2}\|U_1\|^3_{\mathfrak{E}_\omega}.
\end{equation*}
These terms are also perturbative.

Finally, we compute $\chi^{+}(D)\mathcal{Q}(U_1,U_2)$.
\begin{equation*}
    \chi^{+}(D)\mathcal{Q}(U_1,U_2) = \chi^{+}(D)\mathcal{Q}(U_1, F(U_1)) + \underline{\mathcal{O}}_{+}(\epsilon^{2} \| U_1\|^3_{\mathfrak{E}_\omega}).
\end{equation*}
Recall that 
\begin{equation*}
F(U_1) = -L(D)^{-1}\left[ \frac{\gamma^2}{2}U_1^2+gU_1|D|U_1 + \frac{g}{2}|D|U_1^2 -\frac{\sigma}{2}|D|(|D|U_1)^2\right], 
\end{equation*}
Direct computation using Lemma \ref{t:ApproximateIdentity} gives
\begin{align*}
&\chi^{+}(D)\mathcal{Q}(U_1, F(U_1))\\
 =& \sigma\chi^{+}(D)\left(|D|(|D|U_1|D|F(U_1))\right) - \gamma^2 \chi^{+}(D)\left(U_1F(U_1)  \right) - g\chi^{+}(D)(U_1|D|F(U_1))\\
&- g\chi^{+}(D)(F(U_1)|D|U_1)-g\chi^{+}(D)(|D|(F(U_1)U_1))\\
=& -\frac{\sigma \gamma^2}{2}\chi^{+}(D)\left(|D|(|D|U_1^-|D|L(D)^{-1}(U_1^+)^2\right) - g\sigma \chi^{+}(D)\left(|D|(|D|U_1^-|D|L(D)^{-1}(U_1^+|D|U_1^+)\right)\\
&-\frac{g\sigma}{2}\chi^{+}(D)\left(|D|(|D|U_1^-|D|^2L(D)^{-1}(U_1^+)^2)\right)+ \frac{\sigma^2}{2}\chi^{+}(D)\left(|D|(|D|U_1^-|D|^2L(D)^{-1}(|D|U_1^+)^2\right)\\
&+\frac{\gamma^4}{2} \chi^+(D)\left(U_1^-L(D)^{-1}(U_1^+)^2\right) + g\gamma^2 \chi^{+}(D)(U_1^- L(D)^{-1}(U_1^+ |D|U_1^+))\\
&+ \frac{g\gamma^2}{2}  \chi^{+}(D)(U_1^- L(D)^{-1}|D|(U_1^+)^2)-\frac{\sigma \gamma^2}{2} \chi^{+}(D)(U_1^- L(D)^{-1}|D|(|D|U_1^+)^2) \\
&+ \frac{g\gamma^2}{2}\chi^+(D)(U_1^- |D|L(D)^{-1}(U_1^+)^2)+g^2\chi^+(D)(U_1^- |D|L(D)^{-1}(U_1^+ |D|U_1^+)) \\
&+ \frac{g^2}{2}\chi^+(D)(U_1^-|D|^2L(D)^{-1}(U_1^+)^2)-\frac{g\sigma}{2}\chi^+(D)(U_1^- |D|^2L(D)^{-1}(|D|U_1^+)^2) \\
&+ \frac{g\gamma^2}{2}\chi^+(D)(L(D)^{-1}(U_1^+)^2|D|U_1^-) +g^2\chi^+(D)(L(D)^{-1}(U_1^+|D|U_1^+) |D|U_1^-) \\
&+ \frac{g^2}{2}\chi^+(D)(L(D)^{-1}|D|(U_1^+)^2 |D|U_1^-) -\frac{g\sigma}{2}\chi^+(D)(L(D)^{-1}|D|(|D|U_1^+)^2 |D|U_1^-) \\
&+ \frac{g\gamma^2}{2} \chi^{+}(D)(|D|(L(D)^{-1}(U_1^+)^2U_1^-)) + g^2\chi^+(D)(|D|(L(D)^{-1}(U_1^+ |D|U_1^+)U_1^-)) \\
&+ \frac{g^2}{2}\chi^+(D)(|D|(L(D)^{-1}|D|(U_1^+)^2 U_1^-))-\frac{g\sigma}{2}\chi^+(D)(|D|(|D|L(D)^{-1}(|D|U_1^+)^2U_1^-)) \\
&+ \gamma^4 \ell^{-1}(0) \chi^+(D)\left(|U_1^+|^2U_1^+\right) +g\gamma^2\ell^{-1}(0) \chi^+(D)((U_1^+)^2 |D|U_1^-)\\
& +g\gamma^2\ell^{-1}(0) \chi^+(D)( |U_1^+|^2 |D|U_1^+) +g\gamma^2\ell^{-1}(0)\chi^+(D)(|U_1^+|^2|D|U_1^+)\\ 
&+g^2\ell^{-1}(0)\chi^+(D)(U_1^+||D|U_1^+|^2) +g^2\ell^{-1}(0)\chi^+(D)(U_1^-(|D|U_1^+)^2)\\
&+g\gamma^2\ell^{-1}(0)\chi^+(D)|D|(|U^+_1|^2U^+_1) + g^2\ell^{-1}(0)\chi^+(D)|D|((U_1^+)^2|D|U_1^-)\\
&+g^2\ell^{-1}(0)\chi^+(D)|D|(|U_1^+|^2|D|U_1^+)+ \underline{\mathcal{O}}_{+}(\epsilon^{2} \| U_1\|^3_{\mathfrak{E}_\omega}).
\end{align*}
Hence, we have
\begin{align*}
&\chi^{+}(D)\mathcal{Q}(U_1, F(U_1))\\
=& -\sigma \gamma^2|\omega|^3\ell^{-1}(2|\omega|)\chi^+(D)\left(|U_1^+|^2U_1^+\right) -2g\sigma |\omega|^4 \ell^{-1}(2|\omega|)\chi^+(D)\left(|U_1^+|^2U_1^+\right)\\
&- 2g\sigma |\omega|^4\ell^{-1}(2|\omega|)\chi^+(D)\left(|U_1^+|^2U_1^+\right)
+2\sigma^2 |\omega|^6\ell^{-1}(2|\omega|)\chi^+(D)\left(|U_1^+|^2U_1^+\right)\\
&+\frac{\gamma^4}{2}\ell^{-1}(2|\omega|) \chi^+(D)\left(|U_1^+|^2U_1^+\right) + g\gamma^2 |\omega| \ell^{-1}(2|\omega|)\chi^+(D)\left(|U_1^+|^2U_1^+\right) \\
&+ g\gamma^2 |\omega| \ell^{-1}(2|\omega|)\chi^+(D)\left(|U_1^+|^2U_1^+\right) -\sigma \gamma^2 |\omega|^3 \ell^{-1}(2|\omega|)\chi^+(D)\left(|U_1^+|^2U_1^+\right) \\
&+ g\gamma^2 |\omega| \ell^{-1}(2|\omega|)\chi^+(D)\left(|U_1^+|^2U_1^+\right) +2g^2 |\omega|^2 \ell^{-1}(2|\omega|)\chi^+(D)\left(|U_1^+|^2U_1^+\right) \\
&+2g^2 |\omega|^2 \ell^{-1}(2|\omega|)\chi^+(D)\left(|U_1^+|^2U_1^+\right)-2g\sigma |\omega|^4 \ell^{-1}(2|\omega|)\chi^+(D)\left(|U_1^+|^2U_1^+\right) \\
&+ \frac{g\gamma^2}{2}|\omega|\ell^{-1}(2|\omega|)\chi^+(D)\left(|U_1^+|^2U_1^+\right) +g^2 |\omega|^2\ell^{-1}(2|\omega|)\chi^+(D)\left(|U_1^+|^2U_1^+\right) \\
&+g^2 |\omega|^2\ell^{-1}(2|\omega|)\chi^+(D)\left(|U_1^+|^2U_1^+\right) -g\sigma |\omega|^4 \ell^{-1}(2|\omega|)\chi^+(D)\left(|U_1^+|^2U_1^+\right) \\
&+ \frac{g\gamma^2}{2}|\omega|\ell^{-1}(2|\omega|)\chi^+(D)\left(|U_1^+|^2U_1^+\right) +g^2 |\omega|^2\ell^{-1}(2|\omega|)\chi^+(D)\left(|U_1^+|^2U_1^+\right) \\
& +g^2 |\omega|^2\ell^{-1}(2|\omega|)\chi^+(D)\left(|U_1^+|^2U_1^+\right) -g\sigma |\omega|^4 \ell^{-1}(2|\omega|)\chi^+(D)\left(|U_1^+|^2U_1^+\right) \\
&+ \gamma^4 \ell^{-1}(0) \chi^+(D)\left(|U_1^+|^2U_1^+\right) +g\gamma^2 |\omega| \ell^{-1}(0)\chi^+(D)\left(|U_1^+|^2U_1^+\right)  \\
&+g\gamma^2 |\omega| \ell^{-1}(0)\chi^+(D)\left(|U_1^+|^2U_1^+\right) +g\gamma^2 |\omega| \ell^{-1}(0)\chi^+(D)\left(|U_1^+|^2U_1^+\right) \\
&+g^2 |\omega|^2 \ell^{-1}(0)\chi^+(D)\left(|U_1^+|^2U_1^+\right) +g^2 |\omega|^2 \ell^{-1}(0)\chi^+(D)\left(|U_1^+|^2U_1^+\right) \\
&+g\gamma^2 |\omega| \ell^{-1}(0)\chi^+(D)\left(|U_1^+|^2U_1^+\right) +g^2 |\omega|^2 \ell^{-1}(0)\chi^+(D)\left(|U_1^+|^2U_1^+\right) \\
 &+g^2 |\omega|^2 \ell^{-1}(0)\chi^+(D)\left(|U_1^+|^2U_1^+\right) + \underline{\mathcal{O}}_{+}(\epsilon^{2} \| U_1\|^3_{\mathfrak{E}_\omega}) \\
=& -2\sigma \gamma^2|\omega|^3\ell^{-1}(2|\omega|)\chi^+(D)\left(|U_1^+|^2U_1^+\right) -8g\sigma |\omega|^4 \ell^{-1}(2|\omega|)\chi^+(D)\left(|U_1^+|^2U_1^+\right) \\
&+ 4g\gamma^2 |\omega| \ell^{-1}(2|\omega|)\chi^+(D)\left(|U_1^+|^2U_1^+\right) 
 +8g^2 |\omega|^2 \ell^{-1}(2|\omega|)\chi^+(D)\left(|U_1^+|^2U_1^+\right) \\
 &+2\sigma^2 |\omega|^6\ell^{-1}(2|\omega|)\chi^+(D)\left(|U_1^+|^2U_1^+\right) +\frac{\gamma^4}{2}\ell^{-1}(2|\omega|) \chi^+(D)\left(|U_1^+|^2U_1^+\right) \\
&  + \gamma^4 \ell^{-1}(0)\chi^+(D)\left(|U_1^+|^2U_1^+\right) +4g\gamma^2 |\omega| \ell^{-1}(0)\chi^+(D)\left(|U_1^+|^2U_1^+\right) \\
&+4g^2 |\omega|^2 \ell^{-1}(0)\chi^+(D)\left(|U_1^+|^2U_1^+\right)  + \underline{\mathcal{O}}_{+}(\epsilon^{2} \| U_1\|^3_{\mathfrak{E}_\omega}).
\end{align*}
Since $\omega$ is the only real root of the quadratic equation
\begin{equation*}
    \sigma k^2 +c_*^2 k + (g+c_*\gamma) =0, \quad k\leq 0,
\end{equation*}
we get $|\omega|= \frac{c_*^2}{2\sigma}$, so that
\begin{equation*}
    c_* = \pm \sqrt{2\sigma |\omega|}, \quad \ell(0) = \ell(2|\omega|) = g+c_* \gamma>0.
\end{equation*}
Here the sign for $c_*$ depends whether $c_*$ is $c_1$ or $c_2$; see Figure \ref{f:Cplusminus}.
Therefore, we get the approximation relation \eqref{ApproximateTwo}.
\end{proof}

For the cubic terms, we have the following formula for the leading term on the right-hand side of~\eqref{UOneReduced}:
\begin{lemma}
Let $U_1\in X_1$, $U_2\in X_2$, and $U = U_1 + U_2$; then
\begin{equation}
\begin{aligned}
\chi^+&(D)\Big(- \frac{\sigma}{2}\partial_\alpha (U_\alpha)^3+\sigma \left((|D|U)^2U_\alpha\right)_\alpha - \sigma |D|\left(|D|UU_\alpha^2\right)-  \frac{\gamma^2}{6}|D|U^3+\frac{\gamma^2}{2}U|D|U^2 \\
&- \frac{\gamma^2}{2}U^2|D|U \Big) 
=   \left(-\frac{\sigma}{2} |\omega|^3-\gamma^2\right)|\omega|\chi^+(D)\left(|U_1^+|^2U_1^+\right)  + \underline{\mathcal{O}}_{+}(\epsilon^{2} \|U_1\|^3_{\mathfrak{E}_\omega}).   
\end{aligned}  \label{ApproximateThree}
\end{equation}
\end{lemma}
\begin{proof}
Each of these terms on the left of \eqref{ApproximateThree} is of $\chi^+(D)\mathcal{C}(U,U,U)$ type, where $\mathcal{C}(U,U,U)$ is a trilinear form in $U$ with at most four derivatives.
When at least one of the arguments in $\mathcal{C}(U,U,U)$ is $U_2$,
\begin{align*}
&\|\mathcal{C}(U_1, U_1, U_2)\|_{H^2} + \|\mathcal{C}(U_1, U_2, U_2)\|_{H^2} +   \|\mathcal{C}(U_2, U_2, U_2)\|_{H^2}\lesssim \epsilon^{2}\|U_1\|^4_{\mathfrak{E}_\omega}, \\
&\|d\mathcal{C}(U_1, U_1, U_2)\|_{\mathcal{L}(\mathcal{X}_1, L^2)} + \|d\mathcal{C}(U_1, U_2, U_2)\|_{\mathcal{L}(\mathcal{X}_1, L^2)} + \|d\mathcal{C}(U_2, U_2, U_2)\|_{\mathcal{L}(\mathcal{X}_1, L^2)}\lesssim \epsilon^{2}\|U_1\|^3_{\mathfrak{E}_\omega}.
\end{align*}
These terms are thus perturbative.
It suffices to estimate $\chi^+(D)\mathcal{C}(U_1, U_1, U_1)$.

Recall that $U_1=U_1^++U_1^-$ with $U_1^-=\overline{U_1^+}$ and note that a product of three $U_1^{\pm}$ lies in the frequency support of $\chi^+(D)$ only when there are exactly two $U_1^+$ factors and one $U_1^-$ factor.
Hence, using the results in Lemma \ref{t:ApproximateIdentity}, one can write
\begin{align*}
\chi^+(D)\left(-\frac{\sigma}{2}\partial_\alpha(U_{1,\alpha})^3\right)&=- \chi^+(D)\left(\frac{3\sigma}{2}\partial_\alpha((U_{1,\alpha}^{+})^2U_{1,\alpha}^{-})\right) \\
&= \frac{3\sigma}{2}|\omega|^4\chi^{+}(D)(|U^+_1|^2U^+_1) + \underline{\mathcal{O}}_{+}(\epsilon^{2} \|U_1\|^3_{\mathfrak{E}_\omega}), \\
\sigma\chi^+(D) \left((|D|U_1)^2U_{1,\alpha}\right)_\alpha & = \sigma\chi^+(D)\left((|D|U^+_1)^2U^{-}_{1,\alpha}\right)_\alpha + 2\sigma\chi^+(D)\left(U_{1,\alpha}^+|D|U_1^+ |D|U_1^-\right)_\alpha \\ 
&= -\sigma|\omega|^4\chi^+(D)\left(|U_1^+|^2U_1^+\right)+\underline{\mathcal O}_+\left(\epsilon^{2}\|U_1\|_{\mathfrak{E}_\omega}^3\right),\\
\chi^+(D)\left(- \sigma |D|\left(|D|U_1U_{1,\alpha}^2\right)\right) & = -\sigma\chi^+(D)\left(|D|\left(|D|U_1^-(U_{1,\alpha}^+)^2\right)\right) -2\sigma\chi^+(D)\left(|D|\left(|D|U_1^+U_{1,\alpha}^+U_{1,\alpha}^-\right)\right) \\ 
&= -\sigma|\omega|^4\chi^+(D)\left(|U_1^+|^2U_1^+\right)+\underline{\mathcal O}_+\left(\epsilon^{2}\|U_1\|_{\mathfrak{E}_\omega}^3\right),\\
\chi^+(D)\left(\frac{\gamma^2}{2}U_1|D|(U_1)^2\right) & = \gamma^2\chi^+(D)\left(U_1^+|D|(U_1^+ U_1^-)\right) + \frac{\gamma^2}{2}\chi^+(D)\left(U_1^-|D|(U_1^+)^2\right) \\ 
&= \gamma^2|\omega|\chi^+(D)\left(|U_1^+|^2U_1^+\right)+\underline{\mathcal O}_+\left(\epsilon^{2}\|U_1\|_{\mathfrak{E}_\omega}^3\right),\\
\chi^+(D)\left(-\frac{\gamma^2}{2}U_1^2|D|U_1\right)
&=-\frac{3\gamma^2}{2}|\omega| \chi^+(D)\left(|U_1^+|^2U_1^+\right)+\underline{\mathcal O}_+\left(\epsilon^{2}\|U_1\|^3_{\mathfrak{E}_\omega}\right),\\
\chi^+(D)\left(-\frac{\gamma^2}{6}|D|U_1^3\right) & =  -\frac{\gamma^2}{2}\chi^+(D)\left(|D|(U_1^+ U_1^+ U_1^-)\right)\\
&=-\frac{\gamma^2}{2}|\omega| \chi^+(D)\left(|U_1^+|^2U_1^+\right) +\underline{\mathcal O}_+\left(\epsilon^{2}\|U_1\|^3_{\mathfrak{E}_\omega}\right).  
\end{align*}
Collecting the above estimates gives the proof of \eqref{ApproximateThree}.
\end{proof}
As for the remainder term $\chi^+(D)\tilde{R}(U)$, we use \eqref{RemainderCapillaryTwo} to estimate
\begin{equation}
 \chi^+(D)\tilde{R}(U)   = \underline{\mathcal{O}}_{+}(\epsilon^{2} \|U_1\|^4_{\mathfrak{E}_\omega}). \label{TildeRUBound}
\end{equation}

Substituting equations \eqref{ApproximateOne}, \eqref{ApproximateTwo}, \eqref{ApproximateThree}, and \eqref{TildeRUBound} back to \eqref{UOneReduced}, the frequency-localized Babenko equation is simplified to 
\begin{equation}
\begin{aligned}
&\left(g+ c_{*}\gamma-c_{*}^2|D|+ \sigma |D|^2\right)U_1^+  + \epsilon^2a_1 U_1^+  \\
=&a_2\chi^+(D)\left(|U_1^+|^2U_1^+\right)+  \underline{\mathcal{O}}_{+}(\epsilon^4 \| U_1\|_{\mathfrak{E}_\omega})+\underline{\mathcal{O}}_{+}(\epsilon^{2} \| U_1\|^3_{\mathfrak{E}_\omega}),
\end{aligned} \label{UPlusOneReduced}
\end{equation}
where the coefficients $a_1$ and $a_2$ are given by
\begin{equation} \label{AOneTwoDef}
    a_1 = \sqrt{\gamma^2-4\omega(g+\sigma \omega^2)}, \quad a_2 = a_3 -\frac{\sigma}{2} |\omega|^4-\gamma^2 |\omega|,
\end{equation}
for the constant $a_3$ defined in \eqref{aThreeDef}.

The constant $a_1$ is positive; we now consider the positivity of $a_2$.
One can write
\begin{align*}
    a_2 =& (g+c_* \gamma)^{-1}\left(2\sigma^2 |\omega|^6 -8g\sigma |\omega|^4-2\sigma \gamma^2|\omega|^3 + 12g^2 |\omega|^2 +8g\gamma^2 |\omega| + \frac{3}{2}\gamma^4 \right)  -\frac{\sigma}{2} |\omega|^4-\gamma^2 |\omega| \\
    =& (g+c_* \gamma)^{-1}\left(2\sigma^2 |\omega|^6 -8g\sigma |\omega|^4-2\sigma \gamma^2|\omega|^3 + 12g^2 |\omega|^2 +8g\gamma^2 |\omega| + \frac{3}{2}\gamma^4 \right) \\
     &- (g+c_* \gamma)^{-1} \left( \frac{g\sigma}{2} |\omega|^4 +g\gamma^2 |\omega|+ \frac{c_*\gamma\sigma}{2} |\omega|^4 +c_*\gamma^3 |\omega|\right) \\
     =& (g+c_* \gamma)^{-1}\left(2\sigma^2 |\omega|^6 -\frac{17}{2}g\sigma |\omega|^4-2\sigma \gamma^2|\omega|^3 + 12g^2 |\omega|^2 +7g\gamma^2 |\omega| + \frac{3}{2}\gamma^4 \right) \\
     &-(g+c_*\gamma)^{-1}(c_*\gamma)\left(\frac{\sigma}{2} |\omega|^4 +\gamma^2|\omega|\right).
\end{align*}

We consider the two cases $g>0$ and $g=0$ separately.
\begin{enumerate}
    \item \textbf{The gravity-capillary case where $g>0$}. 
    We first consider the critical velocity $c_*$ such that $-g<c_*\gamma<0$.
    This corresponds to choosing the frequency $\omega$ that is closer to the origin between $\omega_1$ and $\omega_2$.
    The positivity of the other possible value of $a_2$ in the case $g>0$ is proven in Section~\ref{s:parameterRegimes}.

    Recall that we have
    \begin{equation*}
        c_*^4=4\sigma(g+c_*\gamma), \quad  |\omega|=\frac{c_*^2}{2\sigma},
    \end{equation*}
    so that
    \begin{equation*}
        |\omega|^2=(g+c_*\gamma)\sigma^{-1}.
    \end{equation*}
    Thus $|\omega|<\sqrt{g/\sigma}$ when $c_*\gamma<0$, whereas $|\omega|>\sqrt{g/\sigma}$ for the other critical velocity, for which $c_*\gamma>0$.
    This leads to the estimate
    \begin{align*}
        &a_2> (g+c_* \gamma)^{-1}\left(2\sigma^2 |\omega|^6 -\frac{17}{2}g\sigma |\omega|^4-2\sigma \gamma^2|\omega|^3 + 12g^2 |\omega|^2 +7g\gamma^2 |\omega| + \frac{3}{2}\gamma^4 \right) \\  
          =& \left(2\sigma |\omega|^4-\frac{17}{2}g|\omega|^2-2\gamma^2|\omega|\right)+(g+c_* \gamma)^{-1}\left(12g^2|\omega|^2+7g\gamma^2|\omega|+\frac{3}{2}\gamma^4\right)\\
          =&\left(2(g+c_*\gamma)-\frac{17}{2}g+12g^2\left(g+c_*\gamma\right)^{-1}\right)|\omega|^2+(g+c_* \gamma)^{-1}\left((5g-2c_*\gamma)\gamma^2|\omega|+\frac{3}{2}\gamma^4\right).
    \end{align*}
    Noting that 
    \begin{align*}
     &2(g+c_*\gamma)-\frac{17}{2}g+12g^2\left(g+c_*\gamma\right)^{-1} \\
     \geq& 2\sqrt{2(g+c_*\gamma)
     \cdot 12g^2\left(g+c_*\gamma\right)^{-1}} -\frac{17}{2}g \geq \left(4\sqrt{6}- \frac{17}{2} \right)g >0,
     \end{align*}
     and
     \begin{equation*}
     (5g-2c_*\gamma)\gamma^2|\omega|+\frac{3}{2}\gamma^4>0,
    \end{equation*}
    we conclude that $a_2 >0$.
    
    \item \textbf{The pure capillary case where $g= 0$}. In this case, we need $c_*\gamma>0$, so that only $c_*=\sqrt[3]{4\sigma\gamma}$ works, and we can write
    \begin{align*}
        a_2 =& (c_* \gamma)^{-1}\left(2\sigma^2 |\omega|^6 -\frac{c_* \gamma \sigma}{2}|\omega|^4 -2\sigma \gamma^2|\omega|^3 -c_*\gamma^3 |\omega|  + \frac{3}{2}\gamma^4\right) \\
        =& (c_* \gamma)^{-1}\left( \frac{1}{2^{5}\sigma^{4}}c_*^{12} - \frac{\gamma}{2^5 \sigma^3} c_*^9 -\frac{\gamma^2}{2^2\sigma^2}c_*^6 -\frac{\gamma^3}{2\sigma}c_*^3 +\frac{3}{2}\gamma^4\right) \\
        =& \frac{\gamma^4}{c_* \gamma} \left( \frac{1}{2^5}\left(\frac{c_*^3}{\sigma \gamma} \right)^4 - \frac{1}{2^5}\left(\frac{c_*^3}{\sigma \gamma} \right)^3 - \frac{1}{2^2}\left(\frac{c_*^3}{\sigma \gamma} \right)^2 - \frac{1}{2}\frac{c_*^3}{\sigma \gamma} + \frac{3}{2}\right) \\
        =& \frac{\gamma^4}{c_* \gamma} \left( \frac{4^4}{2^5} - \frac{4^3}{2^5} - \frac{4^2}{2^2} - \frac{4}{2} + \frac{3}{2}\right) >0,
    \end{align*}
    since in this case we have  $|\omega| = \frac{c_*^2}{2\sigma}$.
\end{enumerate}

Finally, for a small positive constant $\epsilon$, we rescale $U_1^+$,
\begin{equation*}
U_1^{+}(\alpha) = \frac{1}{2}\epsilon \rho(\epsilon \alpha) e^{-i\omega\alpha},
\end{equation*}
and substitute $U_1^{+}$ into the equation \eqref{UPlusOneReduced}, where the frequency $\omega$ is either $\omega_1$ or $\omega_2$ depending on whether the velocity $c_*$ is  $c_1$ or $c_2$.
Here, the Fourier transform of $\rho$ is supported in $\left(-\frac{\delta}{\epsilon},\frac{\delta}{\epsilon}\right)$.

This choice of $U_1^+$ gives $\|U_1^+\|_{\mathfrak{E}_\omega} = \epsilon^{\frac{1}{2}}\|\rho\|_{H^1}$.
It follows that $\rho \in B_R(0) \subseteq \chi_0(\epsilon D)H^2(\mathbb{R})$, where the value $R\leq 2\epsilon^{-\frac{1}{2}}R_1\lesssim 1$.
Let $\beta=\epsilon \alpha$. Then $\rho(\beta)$ solves the equation
\begin{equation}
    \epsilon^{-2}\tilde{L}(|\omega| + \epsilon D ) \rho +  a_1 \rho - a_2 \chi_0(\epsilon D)(|\rho|^2 \rho) + \epsilon^{\frac{1}{2}}\underline{O}_0^\epsilon (\|\rho\|_{H^1}) =0, \label{e:FocusingODE}
\end{equation}
where the symbol $\tilde{\ell}(\xi)$  of $\tilde{L}(D)$ is defined by
\begin{equation}
    \tilde{\ell}(\xi) : = g+ c_{*}\gamma-c_{*}^2 \xi + \sigma \xi^2.  \label{SymbolOfL}
\end{equation} 
$\tilde{\ell}(\xi)$ satisfies the identities
\begin{equation}
    \tilde{\ell}(|\omega|) = 0, \quad \tilde{\ell}^{'}(|\omega|) = 0, \quad \tilde{\ell}^{''}(\xi) = 2\sigma. \label{EllProperty}
\end{equation}

\section{Solving the reduced equation} \label{s:Solving}
In this section, we use the method in Section $5$ of \cite{MR4246394} to solve the reduced equation \eqref{e:FocusingODE} and show that by taking the limit $\epsilon \rightarrow 0$, it converges to a  stationary focusing cubic nonlinear Schr\"odinger equation.

Writing \eqref{e:FocusingODE} as the fixed-point equation
\begin{equation}
 \rho =  \epsilon^2(\tilde{L}(|\omega| + \epsilon D )+ \epsilon^2 a_1)^{-1} \left( a_2 \chi_0(\epsilon D)(|\rho|^2 \rho) + \epsilon^{\frac{1}{2}}\underline{O}_0^\epsilon (\|\rho\|_{H^1}) \right), \label{FixedpointODE}
\end{equation}
we apply the following implicit function theorem for the limit as $\epsilon \rightarrow 0$.

\begin{proposition}[Implicit function theorem II \cite{MR4246394}] \label{t:ImplicitTwo}
Let $\mathcal{X}$ be a Banach space, $X_0$ and $\Lambda_0$ be open neighbourhoods of respectively $x^{*}$ in $\mathcal{X}$ and the origin in $\mathbb{R}$, and $\mathcal{F}: X_0 \times \Lambda_0 \rightarrow \mathcal{X}$  be a function which is differentiable with respect to $x\in X_0$ for each $\lambda \in \Lambda_0$. 
Furthermore, suppose that $\mathcal{F}(x^*, 0) = 0$, $d_1\mathcal{F}[x^*, 0]: \mathcal{X}\rightarrow \mathcal{X}$ is an isomorphism,
\begin{equation*}
 \lim_{x \rightarrow x^*} \|d_1 \mathcal{F}[x,0] - d_1 \mathcal{F}[x^*,0] \|_{\mathcal{L}(\mathcal{X})} =0,
\end{equation*}
and 
\begin{equation*}
\lim_{\lambda \rightarrow 0}\|\mathcal{F}(x, \lambda) - \mathcal{F}(x, 0) \|_{\mathcal{X}} = 0, \quad \lim_{\lambda \rightarrow 0} \|d_1 \mathcal{F}[x, \lambda] - d_1\mathcal{F}[x,0] \|_{\mathcal{L}(\mathcal{X})} =0
\end{equation*}
uniformly over $x \in X_0$. 

There exist open neighbourhoods $x^* \in \mathcal{X}$ and $\Lambda$ of $0$ in $\mathbb{R}$ (with $X \subseteq X_0$, $\Lambda \subseteq \Lambda_0$) and a uniquely determined mapping $h: \Lambda \rightarrow X$  with the properties that
\begin{enumerate}
\item $h$  is continuous at the origin (with $h(0) = x^*)$, 
\item $\mathcal{F}(h(\lambda), \lambda) = 0$ for $\lambda \in \Lambda$, 
\item $x = h(\lambda)$  whenever $(x, \lambda) \in X \times \Lambda$ satisfies $\mathcal{F}(x, \lambda) =0$.
\end{enumerate}
\end{proposition}

To apply the above implicit function theorem to the fixed-point equation \eqref{FixedpointODE}, we first prove the following estimate.
\begin{lemma} \label{t:SymbolApprox}
Let $|\xi| < \frac{\delta}{\epsilon}$; then
\begin{equation*}
 \left|\frac{\epsilon^2}{\tilde{\ell}(|\omega|+ \epsilon \xi)+ \epsilon^2 a_1} - \frac{1}{a_1 + \sigma \xi^2} \right| \lesssim \frac{\epsilon}{(1+\xi^2)^\frac{1}{2}}.
\end{equation*}
\end{lemma}

\begin{proof}
Using direct computation,
\begin{equation*}
 \left|\frac{\epsilon^2}{\tilde{\ell}(|\omega|+ \epsilon \xi)+ \epsilon^2 a_1} - \frac{1}{a_1 + \sigma \xi^2} \right| = \left| \frac{\epsilon^2 \xi^2 \sigma - \tilde{\ell}(|\omega|+ \epsilon \xi)}{(\tilde{\ell}(|\omega|+ \epsilon \xi)+ \epsilon^2 a_1)(a_1 + \sigma \xi^2)} \right|.    
\end{equation*}
By \eqref{EllProperty}, it follows that for the numerator of the right-hand side,
\begin{equation*}
|\epsilon^2 \xi^2 \sigma - \tilde{\ell}(|\omega|+ \epsilon \xi)| \lesssim \epsilon^3 |\xi|^3, \quad |s|\lesssim \delta.
\end{equation*}
On the other hand, for any $\xi\in \mathbb{R}$, $\tilde{\ell}(|\omega|+\epsilon \xi)\gtrsim \xi^2$.
Hence,
\begin{equation*}
 \left|\frac{\epsilon^2}{\tilde{\ell}(|\omega|+ \epsilon \xi)+ \epsilon^2 a_1} - \frac{1}{a_1 + \sigma \xi^2} \right| \lesssim \frac{\epsilon |\xi|^3}{(1+ \xi^2)^2} \lesssim \frac{\epsilon}{(1+\xi^2)^\frac{1}{2}}, \quad |\xi|< \frac{\delta}{\epsilon}.   
\end{equation*}
\end{proof}

Then we recall that it was proved in Section $5.1$ of \cite{MR4246394} that the operator $\chi_0(\epsilon D)$ converges to the identity operator $I$ in the sense that
\begin{equation}
 \lim_{\epsilon \rightarrow 0} \|\chi_0(\epsilon D)- I \|_{\mathcal{L}(H^1(\mathbb{R}), H^{\frac{3}{4}}(\mathbb{R}))} = 0. \label{ChiDLimit}   
\end{equation}
    
Using the approximation result in Lemma \ref{t:SymbolApprox}, one can replace the operator $\epsilon^2(\tilde{L}(|\omega| + \epsilon D )+ \epsilon^2 a_1)^{-1}$ by $(a_1 - \sigma \partial_\beta^2 )^{-1}$ plus some perturbative error terms.
As a consequence, \eqref{FixedpointODE} can be rewritten as
\begin{equation*}
     \rho + H_\epsilon(\rho) =0, 
\end{equation*}
where
\begin{equation*}
H_\epsilon(\rho) = -a_2 (a_1 - \sigma \partial_\beta^2 )^{-1} \chi_0(\epsilon D) \left(|\chi_0(\epsilon D) \rho|^2 \chi_0(\epsilon D) \rho\right) + \epsilon^{\frac{1}{2}} \underline{O}_1^\epsilon (\|\rho\|_{H^1}).
\end{equation*}
Let $X = B_R(0) \subseteq \chi_0(\epsilon D)H^2(\mathbb{R})$, $\Lambda_0 = (-\epsilon_0, \epsilon_0)$  for  sufficiently small positive constant $\epsilon_0$,
\begin{equation*}
\mathcal{X} = H^1_e(\mathbb{R}, \mathbb{C}) = \{\rho \in H^1(\mathbb{R}) : \rho(\beta) = \overline{\rho(-\beta)} \text{ for all } \beta\in \mathbb{R}\},  
\end{equation*}
and $\mathcal{F}(\rho, \epsilon) := \rho + H_{|\epsilon|}(\rho)$.
Here $\epsilon$ can also be negative in $\mathcal{F}(\rho, \epsilon)$, so that $\mathcal{F}(\rho, \epsilon)$ can be defined in an open neighborhood of the origin for $\epsilon$.
Then
\begin{align*}
 \mathcal{F}(\rho, 0) - \mathcal{F}(\rho, \epsilon) =&   a_2 (a_1 - \sigma \partial_\beta^2 )^{-1} [\chi_0(|\epsilon| D) \left(|\chi_0(|\epsilon| D) \rho|^2 \chi_0(|\epsilon| D) \rho\right) - |\rho|^2\rho] + |\epsilon|^{\frac{1}{2}} \underline{O}_1^{|\epsilon|} (\|\rho\|_{H^1}) \\
 =&  a_2 (a_1 - \sigma \partial_\beta^2 )^{-1} \left[\chi_0(|\epsilon| D) \left(|\chi_0(|\epsilon| D) \rho|^2 \chi_0(|\epsilon| D) \rho - |\rho|^2\rho\right) \right]\\
 & + a_2 (a_1 - \sigma \partial_\beta^2 )^{-1} \left[(\chi(|\epsilon|D)-I )|\rho|^2\rho \right] + |\epsilon|^{\frac{1}{2}} \underline{O}_1^{|\epsilon|} (\|\rho\|_{H^1}).
\end{align*}
It follows that by using the limit \eqref{ChiDLimit} and the fact that $H^{\frac{3}{4}}(\mathbb{R})$ is an algebra for the product, one gets for all $\rho \in B_R(0)$ uniformly, 
\begin{equation*}
\lim_{\epsilon \rightarrow 0} \|\mathcal{F}(\rho, \epsilon) - \mathcal{F}(\rho, 0)\|_{H^1} = 0, \quad \lim_{\epsilon \rightarrow 0} \|d_1\mathcal{F}(\rho, \epsilon) - d_1\mathcal{F}(\rho, 0)\|_{\mathcal{L}(H^1(\mathbb{R}, \mathbb{C}))} = 0.
\end{equation*}
Therefore, the fixed-point equation \eqref{FixedpointODE} converges to the equation $\mathcal{F}(\rho, 0) =0$, which is exactly the stationary focusing NLS equation \eqref{CubicShrodingerEqn}:
\begin{equation*}
    \left(a_1-\sigma \partial_{\beta}^2\right)\rho - a_2|\rho|^2 \rho =0, 
\end{equation*}
where positive constants $a_1, a_2$ are defined in \eqref{AOneTwoDef}.
For the existence of solutions of the NLS equation \eqref{CubicShrodingerEqn}, the following result is classical; see, for instance, Section $5.2$ in \cite{MR4246394}.
\begin{proposition}[\hspace{1sp}\cite{MR4246394}]
\begin{enumerate}
    \item  The  stationary focusing cubic nonlinear Schr\"odinger equation \eqref{CubicShrodingerEqn} has precisely two  nontrivial solutions  $\pm \rho_* \in H^1_e(\mathbb{R}; \mathbb{C})$, which are both real-valued even functions.
    \item $d_1\mathcal{F}[\pm \rho_*, 0]: \mathcal{X}\rightarrow \mathcal{X}$ is an isomorphism.
\end{enumerate}
\end{proposition}

Indeed, from \cite{MR0544892}, the explicit solutions of the  stationary focusing NLS equation \eqref{CubicShrodingerEqn} are given by $\{\rho_* e^{i\theta} \}_{\theta \in [0,2\pi)}$, where
\begin{equation}
 \rho_*(\beta) = \sqrt{\frac{2a_1}{a_2}} \text{sech}\left( \sqrt{\dfrac{a_1}{\sigma}}\beta \right).  \label{FormulaRho}
\end{equation}
The choices $\theta = 0$ and $\pi$ give the two real-valued even solutions of \eqref{CubicShrodingerEqn}.
Note that the function $\rho_{*}$ vanishes at infinity.

Using the implicit function theorem Proposition \ref{t:ImplicitTwo}, we get the existence and approximation result:
\begin{proposition}
As long as $a_2>0$, for each sufficiently small value of $\epsilon >0$, the equation \eqref{FixedpointODE}  has two small-amplitude solutions $\rho^\pm_{\epsilon} \in \chi_0(\epsilon D)H^1(\mathbb{R})$ such that $\|\rho^\pm_\epsilon \mp \rho_*\|_{H^1} \rightarrow 0$ as $\epsilon\rightarrow 0$, where $\rho_{*}$ given by \eqref{FormulaRho} is the positive even solution of the stationary focusing NLS equation \eqref{CubicShrodingerEqn}.
\end{proposition}

\begin{remark}
    We have shown already that $a_2>0$ for one of the two ranges of velocities, and since we will show in Section~\ref{s:parameterRegimes} that the other range also satisfies $a_2>0$, the assumption $a_2>0$ will always be satisfied.
\end{remark}

Note that $\chi_0(\epsilon D)H^s(\mathbb{R}), s\geq 0$ are  all topologically equivalent.
Therefore, we have shown that 
\begin{equation*}
    U_1^+(\alpha) = \pm \epsilon \rho_* (\epsilon \alpha) e^{-i\omega \alpha} + o_{H^1}(\epsilon) \in \chi_0(\epsilon D)H^s(\mathbb{R}),\quad  s\geq 0,
\end{equation*}
where $\rho(\beta)$ solves the stationary focusing NLS equation \eqref{CubicShrodingerEqn}.
As a consequence,
\begin{equation*}
    \Im W =U = U_1^+ + U_1^- + U_2 =   \pm \epsilon \rho_* (\epsilon \alpha) (e^{i\omega \alpha} +  e^{-i\omega \alpha}) + o_{H^1}(\epsilon)
\end{equation*}
and $\Im W\in H^s(\mathbb R)$ for any $s\ge 0$,
since by Corollary \ref{t:UTwoHsExist}, $U_2\in H^s$.
To recover $W$ from $\Im W$, we use the fact that $W$ is a holomorphic function, so that
\begin{equation*}
    W = H\Im W + i\Im W = 2i\nP \Im W,
\end{equation*}
which is two times the negative frequency projection of $\Im W$.
Therefore,
\begin{equation*}
     W =   \pm 2\epsilon i \rho_* (\epsilon \alpha)e^{i\omega \alpha} + o_{H^1}(\epsilon) 
\end{equation*}
and $W\in H^s(\mathbb{R})$ for all $s\ge 0$.
The existence of solitary waves follows from Corollary \ref{t:RecoverSolution}.

\section{The range of parameters with two solitary waves}\label{s:parameterRegimes}

In this section, we compute a lower bound for the number of families of gravity-capillary solitary water wave solutions with nonzero vorticity and prove the focusing condition 
\begin{equation*}
    a_2>0 
\end{equation*} 
holds for both velocities $c_1$ and $c_2$.
In this section,  the parameters $g$ and $\sigma$ are positive, and the vorticity $\gamma$ is nonzero.

We will do this by analyzing when the coefficient on the cubic term in~\eqref{UPlusOneReduced} is positive.
We will find an explicit formula for the critical frequencies in terms of the physical parameters of the problem.
An algebraic identity will then show that the stationary cubic NLS equation that the Babenko equation reduces to is always focusing.

To simplify our computations, we rescale the space and time coordinates,
\begin{equation*}
    \tilde{\alpha}= \frac{\gamma^2\alpha }{g},\quad \tilde{t}=\gamma t
\end{equation*}
and choose the rescaled unknowns and velocity
\begin{equation*}
    \tilde{W}=\frac{\gamma^2 W}{g},\quad \tilde{Q}=\frac{\gamma^3 Q}{g^2},\quad \tilde{c}=\frac{\gamma c}{g}.
\end{equation*}

Using the rescaled variables and unknowns and dividing through by the common factor $\frac{g^2}{\gamma^2}$, $\tilde{U} = \Im \tilde{W}(\tilde{\alpha})$ solves the nondimensionalized Babenko equation
\begin{equation}
\begin{aligned}
    &(1+\tilde{c}-\tilde{c}^2|\tilde D|)\tilde{U} -V\left(\frac{\tilde{U}_{\tilde\alpha}}{J^{\frac 1 2}}\right)_{\tilde\alpha}+V|\tilde D|\left(\frac{1+|\tilde D|\tilde{U}}{J^{\frac 1 2}}\right)\\
    &=-\frac{1}{2}\tilde{U}^2-\tilde{U}|\tilde D|\tilde{U}-\frac 1 2 |\tilde D|\tilde{U}^2+\frac 1 2 \tilde{U}|\tilde D|\tilde{U}^2-\frac 1 2 \tilde{U}^2 |\tilde D| \tilde{U}-\frac 1 6 |\tilde D| \tilde{U}^3,\label{e:BabenkoEqnNondim}
\end{aligned}
\end{equation}
where $V$ denotes the dimensionless parameter
\begin{equation*}
    V=\frac{\sigma \gamma^4}{g^3},
\end{equation*}
which measures the relative strength of the vorticity.
Compared to the original Babenko equation \eqref{e:BabenkoEqnLR}, the coefficients of \eqref{e:BabenkoEqnNondim} no longer contain parameters $g, \sigma, \gamma$. 
It is then equivalent to consider the number of families of solutions for \eqref{e:BabenkoEqnNondim} given a fixed positive constant $V$.

To address the question of the sign of the coefficient on the cubic term in~\eqref{UPlusOneReduced} and finish the proof of Theorem~\ref{t:TheoremOne}, it will be easier to work in this dimensionless setting.

The condition $a_2>0$ needed for~\eqref{UPlusOneReduced} to be focusing becomes
\begin{equation}
    (1+\tilde c_*)^{-1}\left(2V^2|\tilde \omega|^6-8V|\tilde \omega|^4-2V|\tilde \omega|^3+12|\tilde \omega|^2+8|\tilde \omega|+\frac 3 2\right)-\frac{V}{2}|\tilde\omega|^4-|\tilde\omega|>0\label{e:focusingConditionNondim}
\end{equation}
in nondimensional coordinates, where the two possible values of $\tilde{\omega}$ in this second condition are the frequencies at which the nondimensionalized wave velocities $\tilde{c}^{\pm}(k)$ have critical points.
In other words, $\tilde{\omega}$'s are the critical points of
\begin{equation*}
    \tilde{c}^{\pm}(k)=\frac{-1\pm \sqrt{1-4k(1+Vk^2)}}{2k}, \quad k<0
\end{equation*}
(See ~\eqref{Defcpm} for comparison).
A direct computation shows these (possible) $\tilde{\omega}$'s are the real roots of the quartic equation
\begin{equation*}
    V^2k^4-2Vk^2+2Vk+1=0, \quad k<0.
\end{equation*}
Letting $G=V^{-1}$ and dividing by $V^2$, we need to find the negative real roots of the quartic polynomial
\begin{equation}
    P(k) = k^4-2Gk^2+2Gk+G^2, \quad k<0.\label{e:quarticG}
\end{equation}

Since the equation is quartic, we can find its roots, and thus the possible values of $\tilde\omega$ by radicals (see e.g.~\cite[Chapter 12]{cox2012galois}).
The four roots of $P(k)$ are given by the roots of the two quadratic equations
\begin{equation}
    x^2-\sqrt{2y_0}x-G+y_0+\frac{G}{\sqrt{2y_0}}=0,\label{e:quadraticNegativeDiscriminant}
\end{equation}
and
\begin{equation}
    x^2+\sqrt{2y_0}x-G+y_0-\frac{G}{\sqrt{2y_0}}=0,\label{e:quadraticPositiveDiscriminant}
\end{equation}
where $y_0$ is any root of the resolvent cubic of $P(k)$, namely, $y_0$ is any root of
\begin{equation*}
    R(y)=8y^3-16Gy^2-4G^2=0.
\end{equation*}

Solving the resolvent cubic by radicals and making the intermediate definitions
\begin{equation}
\begin{aligned}
    z_0&=32G^3+27G^2+3\sqrt 3\sqrt{27G^4+64G^5},\quad z_1=2^{\frac{5}{3}}z_0^{-\frac{1}{3}}G,
\end{aligned}\label{e:intermediateDefinitionsQuartic1}
\end{equation} we see that
\begin{equation*}
    y_0=\frac{2G}{3}\left(1+z_1+z_1^{-1}\right) >0
\end{equation*}
is a real root of $R(y)$.

Since $z_0$ and $z_1$ are always positive for $0<G<\infty$, we have $y_0\ge 2G$.
This means that the discriminant of~\eqref{e:quadraticNegativeDiscriminant} satisfies
\begin{equation*}
    -2y_0+4G-\frac{4G}{\sqrt{2y_0}}\le -\frac{4G}{\sqrt{2y_0}}<0,
\end{equation*}
so that the first quadratic equation has no real roots.
Since $P(0)=G^2>0$, $P(-\sqrt G)=-2G^{\frac{3}{2}}<0$, and $P(k)$ is positive for large enough negative $k$, the intermediate value theorem guarantees that $P(k)$ has two negative real roots.
These must be the roots of~\eqref{e:quadraticPositiveDiscriminant}, so that we have
\begin{equation}
    \tilde\omega_{\pm}=\frac{-\sqrt{2y_0}\pm\sqrt{-2y_0+4G+\frac{4G}{\sqrt{2y_0}}}}{2}\label{e:omegaFormula}
\end{equation}
as the formula for the two critical frequencies.

As shown in Section~\ref{s:Approximate}, the frequency closer to the origin, $\tilde{\omega}_+$, satisfies the focusing condition.
It remains to consider the frequency further from the origin, $\tilde{\omega}_-$.

To simplify the expression~\eqref{e:focusingConditionNondim}, we note that
\begin{equation*}
    (1+\tilde c_*)^{-1}=\frac{4V}{\tilde c_*^4}=\frac{1}{V|\tilde\omega_-|^2},
\end{equation*}
so that it suffices to show
\begin{equation}
    \left(2V|\tilde \omega_-|^4-8|\tilde \omega_-|^2-2|\tilde \omega_-|+\frac{12}{V}+\frac{8}{V}|\tilde \omega_-|^{-1}+\frac{3}{2V}|\tilde \omega_-|^{-2}\right)-\frac{V}{2}|\tilde\omega_-|^4-|\tilde\omega_-|>0.
\end{equation}
Dividing by $|\tilde\omega_-|>0$, we define
\begin{equation*}
    f(V)=\frac{3V}{2}|\tilde\omega_-(V)|^3-8|\tilde\omega_-(V)|-3+\frac{12}{V}|\tilde\omega_-(V)|^{-1}+\frac{8}{V}|\tilde\omega_-(V)|^{-2}+\frac{3}{2V}|\tilde\omega_-(V)|^{-3}.
\end{equation*}
Completing the square, we obtain
\begin{align*}
    6V|\tilde\omega_-|f(V)
    &=9V^2|\tilde\omega_-|^4-48V|\tilde\omega_-|^2-18V|\tilde\omega_-|+72
      +48|\tilde\omega_-|^{-1}+9|\tilde\omega_-|^{-2}\\
    &=\left(3V|\tilde\omega_-|^2-3|\tilde\omega_-|^{-1}\right)^2
      -48V|\tilde\omega_-|^2+72+48|\tilde\omega_-|^{-1}\\
    &=\left(3V|\tilde\omega_-|^2-8-3|\tilde\omega_-|^{-1}\right)^2+8>0.
\end{align*}
Since $V>0$ and $|\tilde\omega_-|>0$, it follows that $f(V)>0$ for all $V>0$.
Thus the condition~\eqref{e:focusingConditionNondim} holds for both critical frequencies $\tilde{\omega}_\pm$, and the approximating stationary cubic NLS equation is focusing in both cases, giving rise to two families of solitary waves with velocities near $\tilde{c}^\pm$.

This completes the proof of Theorem~\ref{t:TheoremOne}.





\bibliographystyle{plain} 
\bibliography{refs}
\end{document}